\documentclass[a4paper,10pt]{article}
\usepackage{amsmath, amsthm, amssymb, amsfonts, mathrsfs, amscd, relsize, color, graphicx}

\definecolor{highlight}{rgb}{1,0,.6}
\definecolor{extlinkz}{rgb}{0,0,1}
\definecolor{intlinkz}{rgb}{1,0,0}
\definecolor{citlinkz}{rgb}{0,.6,.1}

\usepackage[pdfpagelabels,plainpages=false,colorlinks=true,citecolor=citlinkz,linkcolor=intlinkz,urlcolor=extlinkz,urlbordercolor={0 0 0}]{hyperref}

\newtheorem{theo}{Theorem}
\newtheorem{defi}[theo]{Definition}
\newtheorem{lem}[theo]{Lemma}
\newtheorem{prop}[theo]{Proposition}
\newtheorem{cor}[theo]{Corollary}
\newtheorem{rem}[theo]{Remark}

\newcommand{\wh}[1]{\widehat{#1}}
\newcommand{\wt}[1]{\widetilde{#1}}
\newcommand{\ol}[1]{\overline{#1}}

\newcommand{\R}{\mathbb{R}}
\newcommand{\Z}{\mathbb{Z}}
\newcommand{\T}{\mathbb{T}}
\newcommand{\C}{\mathbb{C}}
\newcommand{\N}{\mathbb{N}}
\newcommand{\SY}{\mathbf{D}}

\newcommand{\Hil}{\mathcal{H}}

\newcommand{\meas}{\mathcal X}
\newcommand{\M}{\mathcal{M}}
\newcommand{\ind}{\mathcal{I}}
\newcommand{\hh}{\mathfrak{h}}

\newcommand{\F}{\mathcal{F}}

\newcommand{\psis}[1]{\langle #1\rangle_\Gamma}

\newcommand{\vsp}{\textnormal{span}}
\newcommand{\Ran}{\textnormal{Ran}}
\newcommand{\Ker}{\textnormal{Ker}}

\newcommand{\vn}{\mathcal{R}}
\newcommand{\vnL}{\mathscr{L}}

\newcommand{\id}{{\mathrm{e}}}

\newcommand{\Id}{\mathbb{I}}
\newcommand{\Proj}{\mathbb{P}}

\newcommand{\iso}{\mathscr T}
\newcommand{\isom}[1]{\iso[#1](x)}
\newcommand{\norm}[1]{\left\|#1\right\|_{\oplus}}
\newcommand{\spr}[1]{\langle #1\rangle_{\oplus}}
\newcommand{\one}[1]{\mathlarger{\chi}_{\raisebox{-.5ex}{$\mathsmaller{#1}$}}}

\newcommand{\lat}{\mathcal{L}}
\newcommand{\LL}{\textnormal{L}}
\newcommand{\wl}{\wh{\LL}}

\title{Spaces invariant under unitary representations of discrete groups}
\author{Davide Barbieri, Eugenio Hern\'andez, Victoria Paternostro}

\begin{document}

\maketitle

\begin{abstract}
We investigate the structure of subspaces of a Hilbert space that are invariant under unitary representations of a discrete group.
We work with square integrable representations, and we show that they are those for which we can construct an isometry intertwining the representation with the right regular representation, that we call a Helson map. We then characterize invariant subspaces using a Helson map, and provide general characterizations of Riesz and frame sequences of orbits. These results extend to the nonabelian setting several known results for abelian groups. They also extend to countable families of generators previous results obtained for principal subspaces.
\end{abstract}

\section{Introduction}

The study of properties of invariant subspaces started with the results of Wiener \cite{Wie32} and Srvinivasan \cite{Sri64} showing that a subspace $V$ of $L^2(\T)$ is invariant under multiplication by exponentials of the form $e^{2 \pi i k x}, k \in \Z$ if and only if $V = \{f \one{E} \, : \, f \in L^2(\T)\}$ for some measurable set $E \subset \T$. The subject is the main object of study of the book of H. Helson \cite{Hel64}.

Strongly connected with these objects are shift-invariant spaces which are  subspaces of $L^2(\R^d)$ invariant under integer translations. Their structure was studied in \cite{BDR94a, BDR94b, RS95, Bow00}. The extension to LCA groups and their countable discrete subgroups was given in \cite{CP10, KT08}, while co-compact subgroups were considered in \cite{BR15}. Other actions than translations were considered in \cite{BHP15b, Ive15}, where the Zak transform is used to study the structure of spaces invariant under the action of an LCA group on a $\sigma$-finite measure space. The setting of compact groups was then treated in \cite{Ive18}.

A general framework that includes the invariant spaces described above is the one that we consider in this paper where we have unitary representations of a countable discrete, not necessarily abelian, group $\Gamma$ on a separable 
Hilbert space $\Hil$.   We will treat the class of square integrable representations, or, equivalently, those for which a bracket map $[\cdot, \cdot] : \Hil \times \Hil \to L^1(\vn(\Gamma))$ can be found (see Definition \ref{noncommutativeBracket}), that are called dual integrable.  Since we shall work in the nonabelian setting, the dual group of $\Gamma$ which plays an important role in the abelian case, will be replaced by the group von Neumann algebra $\vn(\Gamma)$. This approach was started in \cite{BHP15a}.

The purpose of this paper is to study subspaces invariant under dual integrable representations. We will analyze their structure and study the reproducing properties of countable families of orbits.
In the following paragraphs we describe in detail the content and structure of this paper.

After describing in Section \ref{sec:PRELIMINARIES} the tools needed in the paper, we introduce in Section \ref{sec:HELSON} the notion of a Helson map $\iso : \Hil \to L^2((\M,\nu),L^2(\vn(\Gamma)))$ associated to a unitary representation, where $(\M,\nu)$ is a $\sigma$-finite measure space. We prove that the existence of such a map is equivalent to dual integrability. Moreover, a constructive procedure is given to obtain Helson maps from brackets and vice versa.

In Section \ref{sec:LEFT} we study the structure of subspaces of $\ell_2(\Gamma)$ that are invariant under the left regular representation, giving a characterization in Theorem \ref{theo:Fourierinvariance}. This allows us to extend to the noncommutative setting the previously mentioned results of Wiener and Srinivasan.

A characterization of invariant subspaces under a dual integrable representation is given in Section \ref{sec:INVARIANT}, Theorem \ref{theo:multiplicatively}, by means of the Helson map. Such characterization is more explicit for principal invariant subspaces, see Proposition \ref{prop:BDR}, or for finitely generated ones, see Corollary \ref{cor:finitely}. As a consequence, existence of biorthogonal systems of orbits of a single element under a dual integrable representation is characterized by a property of the bracket map in Proposition \ref{prop:minimal}.

Section \ref{sec:REPRODUCING} is dedicated to study reproducing properties of orbits of a countable family of elements of $\Hil$. The reproducing properties we have in mind are those of being Riesz or frame sequences. We will prove existence of Parseval frames of orbits, and characterize families whose orbits generate frames or Riesz sequences.

Several examples are given in Section \ref{sec:EXAMPLES} to illustrate our results:
\begin{enumerate}
\item For the case of integer translations in $L^2(\R)$ the so-called fiberization mapping can be obtained from our Helson maps. \vspace{-1ex}
\item A Helson map is obtained, in the form of a Zak transform, for any representation arising from an action of a discrete group on a $\sigma$-finite measure space.\vspace{-1ex}
\item Subspaces of $\ell_2(\Gamma)$ generated by $f = a \delta_{\gamma_1} + b \delta_{\gamma_2}$ under the left regular representation are studied as an example.\vspace{-1ex}
\item We compute the bracket and a Helson map for the action of the dihedral group $\SY_3$ on $L^2(\R^2)$.\vspace{-1ex}
\item The setting of \cite{BB04, BB18} for translates in number-theoretic groups is shown to fit our general scheme. This allows us to extend the results in \cite{BB18} to several generators.\vspace{-1ex}
\end{enumerate}

\

\noindent
{\bf Acknowledgements}: 
This project has received funding from the European Union's Horizon 2020 research and innovation programme under the Marie Sk\l odowska-Curie grant agreement No 777822. In addition, D. Barbieri and E. Hern\'andez were supported by Grant MTM2016-76566-P (Ministerio de Econom\'ia y Competitividad, Spain), and V. Paternostro was supported by Grants UBACyT 20020170200057BA, CONICET-PIP 11220150100355, MINCyT-PICT 2014-1480 and 2016-2616 (Joven).

\section{Preliminaries}\label{sec:PRELIMINARIES}

The aim of this section is to introduce the basic objects and notations that we will use throughout the paper. We recall here the concept of invariant subspaces, frames and  Riesz sequences. Additionally, we revise   a notion of Fourier duality based on the right regular representation \cite{Kun58, JMP14, BHP15a} and the definition of noncommutative $L^p$ spaces, and provide introductory details on weighted noncommutative $L^2$ spaces.

Some general notation we shall use is the following. The set of all bounded and everywhere defined linear operators on a  Hilbert $\Hil$ will be denoted by 
$\mathcal{B}(\Hil)$ and the subset of $\mathcal{B}(\Hil)$ of unitary 
operator will be denoted by $\mathcal{U}(\Hil)$. For an operator $T$ defined on $\Hil$, not necessarily bounded, we denote by $\Ran(T)$ and $\Ker(T)$ its range and its kernel, respectively. An orthogonal projection onto the closed subspace $W\subset \Hil$ will be denoted by $\Proj_W$.

\subsection{Invariant subspaces}
We will work with subspaces of a Hilbert spaces $\Hil$ that are invariant under the action of a group. To be precise, we start by recalling   that, given $\Gamma$ a countable and discrete group an a Hilbert space $\Hil$, a {\it unitary representation of $\Gamma$ on $\Hil$} is a homomorphism $\Pi:\Gamma\to\mathcal{U}(\Hil)$.

\begin{defi}
Let $\Pi$ be a unitary representation of a discrete and countable group $\Gamma$ on a separable Hilbert space $\Hil$. We say that a closed subspace $V \subset \Hil$ is $\Pi$-invariant if and only if $\Pi(\gamma)V \subset V$ for all $\gamma \in \Gamma$. 
\end{defi}
Given a countable family $\Psi = \{\psi_i\}_{i \in \ind} \subset \Hil$, the closed subspace $V$ defined by 
$V=\ol{\vsp\{\Pi(\gamma)\psi_i \, : \, \gamma \in \Gamma, i \in \ind\}}^\Hil$
is $\Pi$-invariant. It is called the $\Pi$-invariant space generated by $\Psi = \{\psi_i\}_{i \in \ind}$, and we will see that any $\Pi$-invariant subspace is of this form (see e.g. Lemma \ref{lem:generators}). When $\Psi$ contains only one element $\psi$, we will simply use the notation $\psis{\psi} = \ol{\vsp\{\Pi(\gamma)\psi\}_{\gamma \in \Gamma}}^\Hil$
and we call $\psis{\psi}$ {\it principal} $\Pi$-invariant space.

\subsection{Frame and Riesz sequences}
We briefly  recall the definitions of frame and Riesz bases. For a detailed exposition on this subject we refer to \cite{Chr03}.

Let $\mathcal{H}$ be a separable Hilbert space, $\ind$ be  a finite or countable index set and   $\{f_i\}_{i\in \ind}$ be a sequence in $\mathcal{H}.$
The sequence $\{f_i\}_{i\in \ind}$ is said to be  a {\it frame} for $\mathcal{H}$ if there exist $0<A\leq B<+\infty$ such that
\begin{equation*}
A\|f\|^2\leq \sum_{i\in \ind} |\langle f,f_i\rangle|^2\leq B\|f\|^2
\end{equation*}
for all $f\in\mathcal{H}$. The constants $A$ and $B$ are called {\it frame bounds}. When $A=B=1$,  $\{f_i\}_{i\in \ind}$ is called {\it  Parseval frame}.

The sequence $\{f_i\}_{i\in \ind}$ is said to be  a {\it Riesz basis} for $\mathcal{H}$ if it is a complete system in $\mathcal{H}$ and if there exist $0<A\leq B<+\infty$ such that
\begin{equation*}
A\sum_{i\in \ind}|a_i|^2\leq \|\sum_{i\in \ind} a_if_i\|^2\leq B\sum_{i\in \ind}|a_i|^2
\end{equation*}
for all sequences $\{a_i\}_{i\in \ind}$ of finite support.
\newpage

The sequence $\{f_i\}_{i\in \ind}$ is a {\it frame (or Riesz) sequence}, if it is a frame (or Riesz basis) for the Hilbert space it spans, namely $\ol{\vsp\{f_i\}_{ i\in \ind}}^\Hil$.

\subsection{Noncommutative setting}\label{subsec:noncommutative}
Let $\Gamma$ be a discrete and countable group. The right regular representation of $\Gamma$ is the homomorphism $\rho : \Gamma \to \mathcal{U}(\ell_2(\Gamma))$ which acts on the canonical basis of $\ell_2(\Gamma)$, $\{\delta_\gamma\}_{\gamma \in \Gamma}$,   as
$$
\rho(\gamma)\delta_{\gamma'} = \delta_{\gamma'\gamma^{-1}} \quad \gamma, \gamma' \in \Gamma
$$
or, equivalently, such that $\rho(\gamma)f(\gamma') = f(\gamma'\gamma)$ for $f \in \ell_2(\Gamma)$ and 
$\gamma, \gamma'\in\Gamma$. Analogously, the left regular representation is the homomorphism $\lambda : \Gamma \to \mathcal{U}(\ell_2(\Gamma))$ which acts on the canonical basis as
$$
\lambda(\gamma)\delta_{\gamma'} = \delta_{\gamma\gamma'} \quad \gamma, \gamma' \in \Gamma
$$
or, equivalently, such that $\lambda(\gamma)f(\gamma') = f(\gamma^{-1}\gamma')$ for $f \in \ell_2(\Gamma)$
and $\gamma, \gamma'\in\Gamma$.

The right von Neumann algebra of $\Gamma$ is defined as (see e.g. \cite[Section 43, Section 12, Section 13]{Con00} or \cite[Section 3, Section 7]{Tak02})
$$
\vn(\Gamma) = \ol{\vsp\{\rho(\gamma)\}_{\gamma \in \Gamma}}^{\textsc{wot}} ,
$$
where the closure is taken in the weak operator topology (WOT).
The left von Neumann algebra $\vnL(\Gamma)$ of $\Gamma$ is defined analogously in terms of the left regular representation and we recall  that
\begin{equation}\label{eq:commutant}
\vn(\Gamma) = \vnL(\Gamma)' = \{\lambda(\gamma) \, : \, \gamma \in \Gamma\}' = \{\rho(\gamma) \, : \, \gamma \in \Gamma\}''
\end{equation}
where if $\mathcal{S}\subset \mathcal{B}(\Hil)$, $\mathcal{S}'=\{T\in\mathcal{B}(\Hil)\,:\, TS=ST, \,\forall\, S\in\mathcal{S}\}$, the commutant of $\mathcal{S}$.

Given $F \in \vn(\Gamma)$, let $\tau$ be the standard trace given by
$$
\tau(F) = \langle F \delta_\id, \delta_\id\rangle_{\ell_2(\Gamma)} ,
$$
where $\id$ is the identity of $\Gamma$. Recall that $\tau$ is normal, finite and faithful. Moreover, it has the {\it tracial property} which means that $\tau(FG)=\tau(GF)$ for all $F,G\in\vn(\Gamma)$.

For $f,g \in \ell_2(\Gamma)$, the convolution $g \ast f$ is the element of $\ell_\infty(\Gamma)$ given by
\begin{equation}\label{eq:convolution}
g \ast f (\gamma) = \sum_{\gamma' \in \Gamma} f(\gamma') g(\gamma\gamma'^{-1}) = \sum_{\gamma' \in \Gamma} g(\gamma') f(\gamma'^{-1}\gamma), \,\,\,\gamma\in\Gamma .
\end{equation}
By \cite[Proposition 43.10]{Con00}, we have that the elements of the group von Neumann algebra $\vn(\Gamma)$ are bounded convolution operators on $\ell_2(\Gamma)$. More precisely, $F \in \vn(\Gamma)$ if and only if there exists a (unique) convolution kernel $f \in \ell_2(\Gamma)$ such that $F g = g \ast f$. We will use this correspondence as our notion of Fourier duality: for $F \in \vn(\Gamma)$, we will call {\it Fourier coefficients of $F$} the values of its convolution kernel $f$, and denote it with $f=\widehat{F}=\{\wh{F}(\gamma)\}_{\gamma \in \Gamma}$. Therefore
\begin{equation}\label{eq:Risconvolution}
F g = g \ast \wh{F}  \quad \forall \ g \in \ell_2(\Gamma) .
\end{equation}
Note that, by definition of $\tau$ and using (\ref{eq:convolution}), we have 
$$
\wh{F}(\gamma) = \tau(F \rho(\gamma)), \,\,\forall\,\,\gamma\in\Gamma.
$$
Conversely, if $f \in \ell_2(\Gamma)$ is such that $f = \wh{F}$ for some $F \in \vn(\Gamma)$, we will call $F$ the {\it group Fourier transform of $f$}, which is a bounded operator given by
$$
\F_\Gamma f = F = \sum_{\gamma \in \Gamma} f(\gamma) \rho(\gamma)^*
$$
where convergence is intended in the weak operator topology. Observe that they satisfy $f = \wh{\F_\Gamma f}$, or $F = \F_\Gamma \wh{F}$.

Given two operators $F, G \in \vn(\Gamma)$, their composition can be written in terms of this Fourier duality as
\begin{equation}\label{eq:composition}
FG = \F_\Gamma\big(\wh{G} \ast \wh{F}\big).
\end{equation}
Indeed,
\begin{align*}
FG & = \big(\F_\Gamma\wh{F} \big) \big(\F_\Gamma \wh{G} \big) = \Big(\sum_{\gamma' \in \Gamma} \wh{F}(\gamma') \rho(\gamma')^*\Big)\Big(\sum_{\gamma'' \in \Gamma} \wh{G}(\gamma'') \rho(\gamma'')^*\Big)\\
& = \sum_{\gamma',\gamma'' \in \Gamma} \wh{F}(\gamma') \wh{G}(\gamma'') \rho(\gamma''\gamma')^* \ =\ \sum_{\gamma\in \Gamma} \Big(\sum_{\gamma' \in \Gamma}\wh{F}(\gamma') \wh{G}(\gamma\gamma'^{-1}) \Big) \rho(\gamma)^*\\
& = \sum_{\gamma\in \Gamma} (\wh{G}\ast\wh{F})(\gamma) \rho(\gamma)^* = \F_\Gamma( \wh{G}\ast\wh{F}).
\end{align*}

For any $1 \leq p < \infty$ let $\|\cdot\|_p$ be the norm over $\vn(\Gamma)$ given by
$$
\|F\|_p = \tau(|F|^p)^{\frac1p} ,
$$
where $|F|$ is the selfadjoint operator defined by $|F| = \sqrt{F^*F}$ and the $p$-th power is defined by functional calculus of $|F|$. 
Following \cite{Nel74, PX03, BHP15a}, we define the noncommutative $L^p(\vn(\Gamma))$ spaces for $1 \leq p < \infty$ as
$$
L^p(\vn(\Gamma)) = \ol{\vsp\{\rho(\gamma)\}_{\gamma \in \Gamma}}^{\|\cdot\|_p}
$$
while for $p = \infty$ we set $L^\infty(\vn(\Gamma)) = \vn(\Gamma)$ endowed with the operator norm.

A densely defined closed linear operator on $\ell_2(\Gamma)$ is said to be affiliated to $\vn(\Gamma)$ if it commutes with all unitary elements of $\vnL(\Gamma)$. When $p < \infty$, the elements of $L^p(\vn(\Gamma))$ are the linear operators on $\ell_2(\Gamma)$ that are affiliated to $\vn(\Gamma)$ whose $p$-norm is finite.
In particular, for $p < \infty$, the elements of $L^p(\vn(\Gamma))$   are not necessarily bounded, while a bounded operator that is affiliated to $\vn(\Gamma)$ automatically belongs to $\vn(\Gamma)$ as a consequence of von Neumann's Double Commutant Theorem. For $p = 2$ one obtains a separable Hilbert space with scalar product\vspace{-3pt}
$$
\langle F_1, F_2\rangle_2 = \tau(F_2^* F_1)\vspace{-3pt}
$$
for which the monomials $\{\rho(\gamma)\}_{\gamma \in \Gamma}$ form an orthonormal basis.
For these spaces the usual statement of H\"older inequality still holds, so that in particular for any $F \in L^p(\vn(\Gamma))$ with $1 \leq p \leq \infty$ its Fourier coefficients are well defined, and the finiteness of the trace implies that $L^p(\vn(\Gamma)) \subset L^q(\vn(\Gamma))$ whenever $q < p$. Moreover, fundamental results of Fourier analysis such as $L^1(\vn(\Gamma))$ Uniqueness Theorem, Plancherel Theorem between $L^2(\vn(\Gamma))$ and $\ell_2(\Gamma)$ still hold in the present setting (see e.g. \cite[Section 2.2]{BHP15a}). We stress that Plancherel Theorem in this setting extends the usual duality between Fourier transform and Fourier coefficients, turning the two operations into the bounded inverse of one another, between the whole $\ell_2(\Gamma)$ and $L^2(\vn(\Gamma))$.

If $F$ is a closed and densely defined selfadjoint operator that is affiliated to $\vn(\Gamma)$, we will call {\it support of $F$} the spectral projection over the set $\R \setminus \{0\}$. It is the minimal orthogonal projection $s_F$ of $\ell_2(\Gamma)$ such that $F = F s_F = s_F F$, it belongs to $\vn(\Gamma)$ (see e.g. \cite[Theorem 5.3.4]{Sun97}), and it reads explicitly
\begin{equation}\label{eq:support}
s_F = \Proj_{(\Ker(F))^\bot} = \Proj_{\ol{\Ran(F)}} .
\end{equation}

\subsection{Weighted \texorpdfstring{$L^2(\vn(\Gamma))$}{L2} spaces}\label{sec:weight}

This subsection is devoted to define a particular class of spaces that we will use in this paper, which  are called weighted $L^2(\vn(\Gamma))$ spaces. 

\begin{defi}\label{def:qL2}
Let $q \in \vn(\Gamma)$ be an orthogonal projection. We define $q L^2(\vn(\Gamma))$ to be the subspace of $L^2(\vn(\Gamma))$ given by
\begin{equation*}
q L^2(\vn(\Gamma)) := \{qF \, : \, F \in L^2(\vn(\Gamma))\}.
\end{equation*}
Note that this subspace is closed, and that $F \in q L^2(\vn(\Gamma))$ if and only if $F = qF$.
\end{defi}

Given a positive $\Omega \in L^1(\vn(\Gamma))$, let $\hh(\Omega)$ be the subspace of $\vn(\Gamma)$ defined by
$$
\hh(\Omega) := \{F \in \vn(\Gamma) \, : \, s_\Omega F = F\}
$$
where $s_\Omega$ denotes the support of $\Omega$ as defined in (\ref{eq:support}). For $F \in \hh(\Omega)$ define
\begin{equation*}
\|F\|_{2,\Omega} := \|\Omega^\frac12 F\|_2 = \tau(|F^*|^2\Omega)^\frac12.
\end{equation*}
Note that if $F\in\hh(\Omega)$ and $\|F\|_{2,\Omega}=0$, we have that $\Omega^\frac12F=0$ and then $\Ran(F)\subset \Ker(\Omega^\frac12)=\Ker(\Omega)$. This implies that $s_\Omega F=0$ and thus, $F=0$.
As a consequence, it holds that $\|\cdot\|_{2,\Omega}$ is a norm in $\hh(\Omega)$. Its 
  associated scalar product reads
$$
\langle F, G\rangle_{2,\Omega} = \langle \Omega^\frac12 F, \Omega^\frac12 G\rangle_2 = \tau(FG^*\Omega).
$$
\begin{defi}
Given a positive $\Omega \in L^1(\vn(\Gamma))$, we define the weighted space $L^2(\vn(\Gamma),\Omega)$ as the completion of $\hh(\Omega)$ with respect to the $\|\cdot\|_{2,\Omega}$ norm. That is
$$
L^2(\vn(\Gamma),\Omega) = \ol{\hh(\Omega)}^{\|\cdot\|_{2,\Omega}} .
$$
\end{defi}

\begin{prop}\label{prop:weightmap}
Let $\Omega \in L^1(\vn(\Gamma))$ be a positive operator and let $s_\Omega L^2(\vn(\Gamma))$ be as in Definition \ref{def:qL2} for $q=s_\Omega$. Let $\omega : \hh(\Omega) \to s_\Omega L^2(\vn(\Gamma))$ be the mapping defined by
\begin{equation*}
\omega(F) = \Omega^\frac12 F.
\end{equation*}
Then $\omega$ can be extended to a surjective isometry from $L^2(\vn(\Gamma),\Omega)$ onto $s_\Omega L^2(\vn(\Gamma))$.
\end{prop}
\begin{proof}
Observe first that, if $F \in \hh(\Omega) \subset \vn(\Gamma)$, then $\Omega^\frac12 F \in L^2(\vn(\Gamma))$ and $s_\Omega \Omega^\frac12 F = \Omega^\frac12 F$. Thus, $\Omega^\frac12 F \in s_\Omega L^2(\vn(\Gamma))$ and $w$ is well defined. Moreover,
$$
\|\omega(F)\|_2 = \|\Omega^\frac12 F\|_2 = \|F\|_{2,\Omega}.
$$
Thus, $\omega$ extends to an isometry from $L^2(\vn(\Gamma),\Omega)$ to $s_\Omega L^2(\vn(\Gamma))$. To prove surjectivity, take $F_0 \in s_\Omega L^2(\vn(\Gamma))$ such that $F_0 \bot \omega\Big(L^2(\vn(\Gamma),\Omega)\Big)$. Then, in particular, since $s_\Omega \rho(\gamma) \in \hh(\Omega)$ for all $\gamma \in \Gamma$, we have
$$
0 = \langle F_0, \Omega^\frac12 s_\Omega \rho(\gamma)^* \rangle_2 = \langle F_0, \Omega^\frac12 \rho(\gamma)^* \rangle_2 =   \tau(\Omega^\frac12 F_0 \rho(\gamma)) \quad \forall \ \gamma \in \Gamma.
$$
Therefore, $\Omega^\frac12 F_0 = 0$ by $L^1(\vn(\Gamma))$ uniqueness of Fourier coefficients. Hence, $s_\Omega F_0 = 0$, and since $F_0 \in s_\Omega L^2(\vn(\Gamma))$, then $F_0 = 0$, proving surjectivity.
\end{proof}

\begin{rem}\label{rem:extensionomega}
Note that an element $F \in L^2(\vn(\Gamma),\Omega)$ is identified with a Cauchy sequence $\{F_n\}_{n \in \N} \subset \hh(\Omega)$ with respect to the norm $\|\cdot\|_{2,\Omega}$. For any such sequence, $\{\Omega^\frac12 F_n\}_{n \in \N}$ is a Cauchy sequence in $s_\Omega L^2(\vn(\Gamma))$ and then it has a limit in $s_\Omega L^2(\vn(\Gamma))$ that we call $\Omega^\frac12 F$. This is the extension of the isometry $\omega$ to $F \in L^2(\vn(\Gamma),\Omega)$.
\end{rem}

\section{Dual integrability and Helson maps}\label{sec:HELSON}

Let us first recall the definition of bracket map of a unitary representation, as in \cite{BHP15a,HSWW10a}, which is the operator in $L^1(\vn(\Gamma))$ whose Fourier coefficients are $\{\langle \varphi, \Pi(\gamma)\psi\rangle_\Hil\big\}_{\gamma \in \Gamma}$.
\begin{defi}\label{noncommutativeBracket}
Let $\Pi$ be a unitary representation of a discrete and countable group $\Gamma$ on a separable Hilbert space $\Hil$. We say that $\Pi$ is dual integrable if there exists a sesquilinear map $[\cdot,\cdot] : \Hil \times \Hil \to L^1(\vn(\Gamma))$, called bracket map, satisfying
$$
\langle \varphi, \Pi(\gamma)\psi\rangle_\Hil = \tau([\varphi,\psi]\rho(\gamma)) \quad \forall \, \varphi, \psi \in \Hil \, , \ \forall \, \gamma \in \Gamma .
$$
In such a case we will call $(\Gamma,\Pi,\Hil)$ a dual integrable triple.
\end{defi}

Note that, as a consequence of uniqueness of Fourier coefficients in $L^1(\vn(\Gamma))$, the bracket map is unique.

According to \cite[Th. 4.1]{BHP15a}, $\Pi$ is dual integrable if and only if it is square integrable, in the sense that there exists a dense subspace $\mathcal{D}$ of $\Hil$ such that
$$
\big\{\langle \varphi, \Pi(\gamma)\psi\rangle_\Hil\big\}_{\gamma \in \Gamma} \in \ell_2(\Gamma) \quad \forall \, \varphi \in \Hil \, , \ \forall \, \psi \in \mathcal{D}.
$$

Moreover we recall that, by \cite[Prop. 3.2]{BHP15a}, the bracket map satisfies the properties
\begin{itemize}
\item[\textnormal{I)}] $[\psi_1, \psi_2]^* = [\psi_2, \psi_1]$
\item[\textnormal{II)}] $[\psi_1, \Pi(\gamma)\psi_2] = \rho(\gamma)[\psi_1, \psi_2]$ \ , \ \ $[\Pi(\gamma)\psi_1, \psi_2] = [\psi_1, \psi_2]\rho(\gamma)^*$ \ , \ \  $\forall \, \gamma \in \Gamma$
\item[\textnormal{III)}] $[\psi,\psi]$ is nonnegative, and $\|[\psi, \psi]\|_1 = \|\psi\|^2_{\Hil}$
\end{itemize}
for all $\psi, \psi_1, \psi_2 \in \Hil$.

Since, in contrast with \cite{BHP15a},  we are using here a bracket map in terms of the right regular representation, we provide a proof of Property II). By definition of the bracket map and the traciality of $\tau$ we have that for any $\gamma_o\in\Gamma$, 
\begin{align*}
\tau([\psi_1, \Pi(\gamma_0)\psi_2] \rho(\gamma)) & = \langle \psi_1 , \Pi(\gamma) \Pi(\gamma_0) \psi_2 \rangle_{\Hil} = \langle \psi_1 , \Pi(\gamma \gamma_0) \psi_2 \rangle_{\Hil}\\
& = \tau([\psi_1, \psi_2] \rho(\gamma \gamma_0)) = \tau([\psi_1, \psi_2] \rho(\gamma) \rho(\gamma_0))\\
& = \tau(\rho(\gamma_0)[\psi_1, \psi_2] \rho(\gamma)) \, , \quad \forall \, \gamma \in \Gamma .
\end{align*}
Then, by the  $L^1(\vn(\Gamma))$ uniqueness of Fourier coefficients we conclude that $[\psi_1, \Pi(\gamma_0)\psi_2] = \rho(\gamma_0)[\psi_1, \psi_2]$. The other equality is proved from this result and Property I).

Given a $\sigma$-finite measure space $(\M,\nu)$, we denote by $\norm{\varPhi}$ the norm on the Hilbert space $L^2((\M,\nu),L^2(\vn(\Gamma)))$, that reads
$$
\norm{\varPhi} = \left(\int_\M \|\varPhi(x)\|_2^2 d\nu(x)\right)^\frac12 = \left(\int_\M \tau(\varPhi(x)^*\varPhi(x)) d\nu(x)\right)^\frac12
$$
for all $\varPhi \in L^2((\M,\nu),L^2(\vn(\Gamma)))$.

\newpage

\begin{defi}\label{def:Helsonmap}
Let $\Gamma$ be a discrete group and $\Pi$ a unitary representation of $\Gamma$ on the separable Hilbert space $\Hil$. We say that the triple $(\Gamma,\Pi,\Hil)$ admits a Helson map if there exists a $\sigma$-finite measure space $(\M,\nu)$ and an isometry
$$
\iso : \Hil \to L^2((\M,\nu),L^2(\vn(\Gamma)))
$$
satisfying
\begin{equation}\label{eq:Helsonproperty}
\iso[\Pi(\gamma)\varphi] = \iso[\varphi] \rho(\gamma)^* \quad \forall \, \gamma \in \Gamma , \ \forall \, \varphi \in \Hil .
\end{equation}
\end{defi}
Observe that for $\varPsi \in L^2((\M,\nu),L^2(\vn(\Gamma)))$ and $F \in \vn(\Gamma)$ we are denoting with $\varPsi F$ the element of $L^2((\M,\nu),L^2(\vn(\Gamma)))$ that for a.e. $x \in \M$ is given by
\begin{equation}\label{eq:rightaction}
(\varPsi F )(x) = \varPsi(x) F .
\end{equation}

The main theorem of this section is the following.
\begin{theo}\label{theo:DualHelsonEquivalence}
Let $\Gamma$ be a discrete group and $\Pi$ a unitary representation of $\Gamma$ on the separable Hilbert space $\Hil$.Then, the triple $(\Gamma,\Pi,\Hil)$   is dual integrable if and only if  it  admits a Helson map.
\end{theo}

\begin{rem}
It is known that a representation is  square integrable if and only if  it is  unitarily equivalent to a 
 subrepresentation of  a multiple copy of the right regular  representation (see \cite[Prop 4.2]{HM79}). 
In our setting, a Helson map is essentially an isomorphism that implements such unitary equivalence. 

Indeed, given $\Gamma$  a discrete group and $\Pi$ a unitary representation of $\Gamma$ on the separable Hilbert space $\Hil$ with associated  Helson map $\iso$, by a similar argument to the one used in the proof of \cite[Th. 4.1]{BHP15a}, the map
$$
\begin{array}{rcl}
\Gamma \times \iso(\Hil) & \to & \iso(\Hil)\\
(\gamma, \varPhi) & \mapsto & \varPhi \rho(\gamma)^*
\end{array}
$$
defines a unitary representation of $\Gamma$ on $\iso(\Hil)$ that is unitarily equivalent to a summand of a direct integral decomposition of the right  regular representation.

Since we are interested in the structure of such isometry, we provide here a constructive proof of both  implications of  Theorem \ref{theo:DualHelsonEquivalence} in two separate propositions: Proposition \ref{prop:propertyii}, which constructs a bracket map starting from a Helson map, and Proposition \ref{prop:periodization}, which constructs a Helson map starting from a bracket map.

\end{rem}

\begin{prop}\label{prop:propertyii}
 Let $\Gamma$ be a discrete group and $\Pi$ a unitary representation of $\Gamma$ on the separable Hilbert space $\Hil$. Let $(\Gamma,\Pi,\Hil)$ admit a Helson map $\iso$. Then it is a dual integrable triple, and the bracket map can be expressed as
\begin{equation}\label{eq:Helsonbracket}
\displaystyle[\varphi,\psi] = \int_\M \isom{\psi}^* \isom{\varphi} d\nu(x), \quad \forall\,\,\varphi,\psi\in\Hil .
\end{equation}
\end{prop}
\begin{proof}
Let us  first prove that the right hand side of  (\ref{eq:Helsonbracket}) is in $L^1(\vn(\Gamma))$. For this, we only need to see that its norm is finite, which is true because
\begin{align*}
\Big\| & \int_\M \isom{\psi}^* \isom{\varphi} d\nu(x) \Big\|_{1} \leq \int_\M \|\isom{\psi}^* \isom{\varphi}\|_1 d\nu(x)\\
& \leq \int_\M \|\isom{\psi}\|_2 \|\isom{\varphi}\|_2 d\nu(x) \leq \norm{\iso[\psi]} \norm{\iso[\varphi]} = \|\psi\|_\Hil\|\varphi\|_\Hil \end{align*}
where we have used H\"older's inequality on $L^2(\vn(\Gamma))$ and on $L^2(\M,d\nu)$ and the fact that $\iso$ is an isometry. Moreover, since $\iso$ satisfies (\ref{eq:Helsonproperty}), for $\varphi, \phi \in \Hil$ and $\gamma \in \Gamma$, we have
\begin{align*}
\langle & \varphi, \Pi(\gamma)\phi \rangle_\Hil = \spr{\iso[\varphi], \iso[\Pi(\gamma)\phi]} = \int_\M \langle \isom{\varphi}, \isom{\Pi(\gamma)\phi}\rangle_{2} d\nu(x)\\
& = \int_\M \langle \isom{\varphi}, \isom{\phi}\rho(\gamma)^*\rangle_{2} d\nu(x) = \int_\M \tau\Big(\rho(\gamma)\isom{\phi}^*\isom{\varphi}\Big) d\nu(x)\\
& = \tau\bigg(\rho(\gamma) \int_\M \isom{\phi}^* \isom{\varphi} d\nu(x)\bigg)
\end{align*}
where the last identity is due to Fubini's Theorem, which holds by the normality of $\tau$. 
Now, since we have that the Fourier coefficients of $[\varphi,\psi]$ and $\displaystyle\!\!\int_\M \isom{\psi}^* \isom{\varphi} d\nu(x)$ coincide, then (\ref{eq:Helsonbracket}) holds by the $L^1(\vn(\Gamma))$ Uniqueness Theorem.
\end{proof}

We set out to prove the converse of Proposition \ref{prop:propertyii}, to finally prove Theorem \ref{theo:DualHelsonEquivalence}. The following result is needed.

\begin{lem}\label{lem:generators}
Let $\Pi$ be a unitary representation of a discrete and countable group $\Gamma$ on a separable Hilbert space $\Hil$, and let $V \subset \Hil$ be a  $\Pi$-invariant subspace. Then there exists a countable family $\{\psi_i\}_{i \in \ind}$ satisfying $\psis{\psi_i} \bot \psis{\psi_j}$ for $i \neq j$ and such that $V$ decomposes into the orthogonal direct sum
\begin{equation}\label{eq:GramSchmidt}
V = \bigoplus_{i \in \ind} \psis{\psi_i} .
\end{equation}
\end{lem}
\begin{proof}
Let $\{e_n\}_{n \in \N}$ be an orthonormal basis for $V$; choose $\psi_1 = e_1$ and let $V_1 = \psis{\psi_1}$. If $V_1 = V$ the lemma is proved. If $V_1 \neq V$, let $e_{n_2}$ be the first element of $\{e_n\}_{n \in \N}$ such that $e_{n_2} \notin V_1$. Define $\psi_2 = \Proj_{V_1^\bot} e_{n_2}$, where $\Proj_{V_1^\bot}$ stands for the orthogonal projection of $\Hil$ onto $V_1^\bot$ and $V_1^\bot$ is the orthogonal complement of $V_1$ in $V$ (i.e. $V=V_1\oplus V_1^\bot$). It holds that $V_1 \perp \psis{\psi_2}$ since, for $\gamma_1, \gamma_2 \in \Gamma$,
$$
\langle \Pi(\gamma_1) \psi_1, \Pi(\gamma_2)\psi_2\rangle_\Hil = \langle \Pi(\gamma_2^{-1}\gamma_1)\psi_1, \psi_2\rangle_\Hil = 0
$$
because $\Pi(\gamma_2^{-1}\gamma_1)\psi_1 \in V_1$ and $\psi_2 \in V_1^\bot$. Let $V_2 = \psis{\psi_1} \oplus \psis{\psi_2}$. We iterate the process to obtain
$$
V_k = \bigoplus_{j = 1}^k \psis{\psi_j} ,
$$
where $\psis{\psi_i} \perp \psis{\psi_j}$ for $i \neq j$, $i,j = 1, \dots, k$. Since $\{e_1,\dots,e_{n_k}\} \subset V_k$ and $V= \ol{\vsp\{e_n\}_{n \in \N}}^\Hil$, one gets (\ref{eq:GramSchmidt}) after a countable number of steps.
\end{proof}
\begin{rem}
From Lemma \ref{lem:generators} one concludes that any $\Pi$-invariant subspace $V$ of $\Hil$ is generated by a countable family of elements of $V$, namely that $V = \ol{\vsp\{\Pi(\gamma)\psi_i \, : \, \gamma \in \Gamma, i \in \ind\}}$.
\end{rem}

When $\Pi$ is dual integrable, the bracket $[\psi,\psi]$ for nonzero $\psi \in \Hil$ provides a positive $L^1(\vn(\Gamma))$ weight that we can use in order to define the weighted space $L^2(\vn(\Gamma),[\psi,\psi])$ as in Subsection \ref{sec:weight}. Explicitly, the induced norm is
$$
\|F\|_{2,[\psi,\psi]} = \Big(\tau(|F^*|^2[\psi,\psi])\Big)^\frac12 = \|[\psi,\psi]^\frac12 F \|_2
$$
and the inner product is
$$
\langle F_1, F_2\rangle_{2,[\psi,\psi]} = \langle [\psi,\psi]^\frac12 F_1 , [\psi,\psi]^\frac12 F_2\rangle_2 = \tau(F_2^*[\psi,\psi]F_1).
$$
The associated weighted space is needed for the following result, which was proved in \cite[Prop. 3.4]{BHP15a} and lies at the basis of our subsequent constructions. For $\psi \in \Hil$ let us use, in accordance with Subsection \ref{sec:weight}, the notation
$$
\hh = \hh([\psi,\psi]) = \{F \in \vn(\Gamma) \ | F = s_{[\psi,\psi]} F\} .
$$

\begin{prop}\label{prop:isometry}
Let $\Gamma$ be a discrete group and $\Pi$ a unitary representation of $\Gamma$ on the separable Hilbert space $\Hil$
such that $(\Gamma, \Pi,\Hil)$ is a dual integrable triple. Then for any nonzero $\psi \in \Hil$ the map $S_\psi : \vsp\{\Pi(\gamma)\psi\}_{\gamma \in \Gamma} \to \hh$ given by
\begin{equation}\label{eq:isometry}
S_\psi \Big[\sum_{\gamma\in \Gamma} f(\gamma) \Pi(\gamma)\psi \Big] = s_{[\psi,\psi]}\sum_{\gamma\in \Gamma} f(\gamma) \rho(\gamma)^*
\end{equation}
is well-defined and extends to a linear surjective isometry $$S_\psi : \psis{\psi} \to L^2(\vn(\Gamma),[\psi,\psi])$$ satisfying
\begin{equation}\label{eq:intertwining}
S_\psi[\Pi(\gamma)\varphi] = S_\psi[\varphi] \rho(\gamma)^* \, , \quad \forall \, \varphi \in \psis{\psi}.
\end{equation}
\end{prop}
\begin{proof}
Let us first see that $S_\psi$ is well-defined. Suppose that for some $\gamma \in \Gamma$ we have $\Pi(\gamma)\psi = \psi$. Then we need to prove that $s_{[\psi,\psi]} \rho(\gamma)^* = s_{[\psi,\psi]}$. For this, let $v \in \Ran([\psi,\psi])$, and let $u \in \ell_2(\Gamma)$ be such that $v = [\psi,\psi]u$. Then
$$
\rho(\gamma)v = \rho(\gamma)[\psi,\psi]u = [\psi,\Pi(\gamma)\psi]u = [\psi,\psi]u = v,
$$
where we have used Property II) of the bracket map. A simple density argument then ensures that $\rho(\gamma)v = v$ for all $v \in \ol{\Ran([\psi,\psi])}$. This means that $\rho(\gamma) s_{[\psi,\psi]} = s_{[\psi,\psi]}$, and the conclusion follows by taking the adjoint.

Let now $\varphi = \displaystyle\sum_{\gamma \in \Gamma} f(\gamma) \Pi(\gamma)\psi \in \vsp\{\Pi(\gamma)\psi\}_{\gamma \in \Gamma}$ be a finite sum. Then
\begin{align*}
\|S_\psi[\varphi]\|^2_{2,[\psi,\psi]} & = \|[\psi,\psi]^\frac12 \sum_{\gamma\in \Gamma} f(\gamma) \rho(\gamma)^*\|_2^2\\
& = \tau \Big(\sum_{\gamma_1,\gamma_2 \in \Gamma} \ol{f(\gamma_1)} \rho(\gamma_1)[\psi,\psi] f(\gamma_2)\rho(\gamma_2)^*\Big) = \tau \Big([\varphi,\varphi]\Big) = \|\varphi\|_\Hil^2.
\end{align*}
Therefore, $S_\psi$ can be extended by density to a linear isometry from $\psis{\psi}$ to $L^2(\vn(\Gamma),[\psi,\psi])$.
To prove surjectivity, suppose that $F \in L^2(\vn(\Gamma),[\psi,\psi])$ satisfies
$$
\langle F, S_\psi[\varphi]\rangle_{2,[\psi,\psi]} = 0 \quad \forall \ \varphi \in \psis{\psi}.
$$
In particular, for all $\gamma \in \Gamma$
$$
0 = \langle F, S_\psi[\Pi(\gamma)\psi]\rangle_{2,[\psi,\psi]} = \langle F,s_{[\psi,\psi]}\rho(\gamma)^*\rangle_{2,[\psi,\psi]} = \tau(\rho(\gamma)[\psi,\psi]F).
$$
Since both $[\psi,\psi]^\frac12$ and $[\psi,\psi]^\frac12 F$ belong to $L^2(\vn(\Gamma))$, see Remark \ref{rem:extensionomega}, then $[\psi,\psi] F \in L^1(\vn(\Gamma))$ and by the uniqueness of Fourier coefficients one gets $[\psi,\psi]F = 0$. This implies $[\psi,\psi]^\frac12F = 0$, so $\|F\|_{2,[\psi,\psi]} = 0$ and  hence $F = 0$.

Finally, to prove (\ref{eq:intertwining}), it suffices to prove it on a dense subspace. If $\varphi = \displaystyle\sum_{\gamma \in \Gamma} f(\gamma) \Pi(\gamma)\psi \in \vsp\{\Pi(\gamma)\psi\}_{\gamma \in \Gamma}$ is a finite sum, then\vspace{-2ex}
\begin{align*}
S_\psi[\Pi(\gamma)\varphi] & = S_\psi\Big[\sum_{\gamma'\in \Gamma} f(\gamma') \Pi(\gamma\gamma')\psi\Big]
= s_{[\psi,\psi]}\sum_{\gamma'\in \Gamma} f(\gamma') \rho(\gamma\gamma')^* \\
& = s_{[\psi,\psi]} \sum_{\gamma'\in \Gamma}\! f(\gamma') \rho(\gamma')^* \rho(\gamma)^* = S_\psi[\varphi]\rho(\gamma)^* . \qedhere \vspace{-2ex}
\end{align*}
\end{proof}

We are now ready to prove the converse of Proposition \ref{prop:propertyii} to finally get a complete proof of Theorem \ref{theo:DualHelsonEquivalence}.  
\begin{prop}\label{prop:periodization}
Let $\Gamma$ be a discrete group and $\Pi$ a unitary representation of $\Gamma$ on the separable Hilbert space $\Hil$
such that $(\Gamma, \Pi,\Hil)$ is a dual integrable triple. Then $(\Gamma, \Pi,\Hil)$ admits a Helson map.
\end{prop}
\begin{proof}
Let $\Psi=\{\psi_i\}_{i\in\ind}$ be a family as in Lemma \ref{lem:generators} for $\Hil$, i.e.
$\Hil = \displaystyle\bigoplus_{i \in \ind} \psis{\psi_i}$. \vspace{-2ex}\\For $\varphi \in \Hil$ define
$$
U_\Psi(\varphi) = \Big\{[\psi_i,\psi_i]^\frac12 S_{\psi_i}[\Proj_{\psis{\psi_i}}\varphi]\Big\}_{i \in \ind}
$$
where $S_{\psi_i}$ is given by Proposition \ref{prop:isometry} and $\Proj_{\psis{\psi_i}}$ denotes the orthogonal projection of $\Hil$ onto $\psis{\psi_i}$. We shall show that $U_\Psi$ is a Helson map for $(\Gamma,\Pi,\Hil)$ taking values in $\ell_2(\ind,L^2(\vn(\Gamma)))$. For $\varphi \in \Hil$, by Proposition \ref{prop:isometry} we get
\begin{align*}
\|U_\Psi(\varphi)\|^2_{\ell_2(\ind,L^2(\vn(\Gamma)))} & = \sum_{i \in \ind} \|[\psi_i,\psi_i]^\frac12 S_{\psi_i}[\Proj_{\psis{\psi_i}}\varphi]\|^2_2 = \sum_{i \in \ind} \|S_{\psi_i}[\Proj_{\psis{\psi_i}}\varphi]\|^2_{2,[\psi_i,\psi_i]}\\
& = \sum_{i \in \ind} \|\Proj_{\psis{\psi_i}}\varphi\|_\Hil^2 = \|\varphi\|_\Hil^2 .
\end{align*}
This shows that $U_\Psi(\varphi) : \Hil \to \ell_2(\ind,L^2(\vn(\Gamma)))$ is an isometry. Property (\ref{eq:Helsonproperty}) of the Helson map is a consequence of (\ref{eq:intertwining}) and the fact that an orthogonal projection onto an invariant subspace commutes with the representation.
\end{proof}

\section{Left-invariant spaces in $\ell_2(\Gamma)$}\label{sec:LEFT}

In this section we study invariant subspaces of $\ell_2(\Gamma)$ under the left regular representation $\lambda$. As it is customary, we will call such spaces left-invariant.
We begin with the following basic fact.
\begin{lem}\label{lem:rightprojections}
An orthogonal projection onto the closed subspace $V \subset \ell_2(\Gamma)$ belongs to $\vn(\Gamma)$ if and only if $V$ is left-invariant.
\end{lem}
\begin{proof}
By (\ref{eq:commutant})
$\Proj_V$ belongs to $\vn(\Gamma)$ if and only if \ $\Proj_V \lambda(\gamma) = \lambda(\gamma) \Proj_V$ for all $\gamma \in \Gamma$. Let us then first assume that $\lambda(\Gamma)V \subset V$. Then also $V^\bot$ is left-invariant, because for all $\gamma \in \Gamma$, $v \in V$, $v' \in V^\bot$ it holds
$$
\langle v , \lambda(\gamma) v'\rangle = \langle \lambda(\gamma)^*v, v'\rangle = \langle \lambda(\gamma^{-1})v, v'\rangle = 0
$$
so $\lambda(\gamma)v' \bot v$, and hence $\lambda(\Gamma)V^\bot \subset V^\bot$. Then,  for all $u \in \ell_2(\Gamma)$
$$
\Proj_V \lambda(\gamma) u = \Proj_V \lambda(\gamma) \Proj_V u + \Proj_V \lambda(\gamma) \Proj_{V^\bot} u = \lambda(\gamma) \Proj_V u ,
$$
and thus, $\Proj_V \in \vn(\Gamma)$.

Conversely, let $\Proj_V \in \vn(\Gamma)$. Then for all $v \in V$ we have $\lambda(\gamma) v = \lambda(\gamma) \Proj_V v = \Proj_V \lambda(\gamma) v \in V$. Hence, $\lambda(\Gamma)V \subset V$.
\end{proof}

\newpage

For the left regular representation a natural Helson map is provided in the following propositions. 

\begin{prop}\label{prop:Helsonleft}
A Helson map for the left regular representation is the group Fourier transform, that is $\iso : \ell_2(\Gamma) \to L^2(\vn(\Gamma))$ is given by
\begin{equation}\label{eq:Helsonleft}
\iso[f] = \F_\Gamma f = \sum_{\gamma \in \Gamma} f(\gamma) \rho(\gamma)^* \ , \quad f \in \ell_2(\Gamma) 
\end{equation}
where in this case the measure spaces $\M$ is taken to be a singelton. As a consequence, the bracket map for the left regular representation reads
\begin{equation}\label{eq:bracketleft}
[f,g] = (\F_\Gamma g)^* \F_\Gamma f \ , \quad f, g \in \ell_2(\Gamma). 
\end{equation}
\end{prop}
\begin{proof}
By Plancherel Theorem, we have that $\iso$ defined as in (\ref{eq:Helsonleft}) is a surjective isometry. We can check the Helson property (\ref{eq:Helsonproperty}) by direct computation, since
\begin{align}\label{eq:Fourierintertwining}
\F_\Gamma\lambda(\gamma)f & = \sum_{\gamma' \in \Gamma} \lambda(\gamma)f(\gamma') \rho(\gamma')^* = \sum_{\gamma' \in \Gamma} f(\gamma^{-1}\gamma') \rho(\gamma')^* = \sum_{\gamma'' \in \Gamma} f(\gamma'') \rho(\gamma\gamma'')^*\nonumber\\
& = \sum_{\gamma'' \in \Gamma} f(\gamma'') \rho(\gamma'')^*\rho(\gamma)^* = (\F_\Gamma f) \rho(\gamma)^*.
\end{align}
Then, (\ref{eq:bracketleft}) follows from Proposition \ref{prop:propertyii}.
\end{proof}
Analogously, the right regular representation $\rho$ is always dual integrable, and a Helson map $\iso : \ell_2(\Gamma) \to L^2(\vn(\Gamma))$ is provided by
$$
\iso[f] = \sum_{\gamma \in \Gamma} f(\gamma) \rho(\gamma) \ , \quad f \in \ell_2(\Gamma) .
$$

The following theorem characterizes the subspaces of $\ell_2(\Gamma)$ that are invariant under the left regular representation $\lambda$.

\begin{theo}\label{theo:Fourierinvariance}
Let $V \subset \ell_2(\Gamma)$ be a closed subspace. Then the following are equivalent
\begin{itemize}
\item[i)] $V$ is left-invariant;
\item[ii)] $\exists \ q \in \vn(\Gamma)$ orthogonal projection of $\ell_2(\Gamma)$ such that $\F_\Gamma (V) = qL^2(\vn(\Gamma))$. \end{itemize}
Moreover, in this case we have $q = \Proj_V$.
\end{theo}
\begin{proof}
Let us first prove that $i)$ implies $ii)$ Let $q = \Proj_V$, which belongs to $\vn(\Gamma)$ by Lemma \ref{lem:rightprojections}. By (\ref{eq:Risconvolution}) we have $q (f) = f \ast \wh{q}$ for all $f \in \ell_2(\Gamma)$. 
Thus, by \eqref{eq:composition}
\begin{equation}\label{eq:dummy2}
q (\F_\Gamma f)  = \sum_{\gamma \in \Gamma} f \ast \wh{q} \, (\gamma)\rho(\gamma)^*
= \sum_{\gamma \in \Gamma} q (f) (\gamma)\rho(\gamma)^* = \F_\Gamma (q(f)).
\end{equation}

Now, if $f \in V$, then $q(f) = f$ and by \eqref{eq:dummy2}, $\F_\Gamma f = q(\F_\Gamma f)$. So $\F_\Gamma (f) \in qL^2(\vn(\Gamma))$, which shows that $\F_\Gamma (V) \subset q L^2(\vn(\Gamma))$. 
Conversely, if $F \in q L^2(\vn(\Gamma))$, then $qF = F$. If $f = \wh{F}$, we then have that $q (\F_\Gamma f) = \F_\Gamma f$, so by (\ref{eq:dummy2}), $f = q(f) \in V$. Thus $F \in \F_\Gamma(V)$.

Let us prove that $ii)$ implies $i)$ Let $f \in V$. Then $\F_\Gamma f = q G$ for some  $G \in L^2(\vn(\Gamma))$. By (\ref{eq:Fourierintertwining}), we have that, for each $\gamma\in\Gamma$, $\F_\Gamma \lambda(\gamma)f = (\F_\Gamma f) \rho(\gamma)^* = q G \rho(\gamma)^* \in q L^2(\vn(\Gamma))$. This implies that $\lambda(\gamma)f \in V$ for all $\gamma \in \Gamma$.
\end{proof}

The following result extends to general discrete groups a classical result attributed to Srinivasan \cite{Sri64} and Wiener \cite{Wie32} (see also \cite[Corollary 3.9]{BR15}).
\begin{cor}\label{cor:Helson}
Let $W \subset L^2(\vn(\Gamma))$ be a closed subspace. Then
$W \rho(\gamma) \subset W \ \forall\,\gamma \in \Gamma$ if and only if there exists an orthogonal projection $q \in \vn(\Gamma)$ such that $W = q L^2(\vn(\Gamma))$.
\end{cor}
\begin{proof}
By Theorem \ref{theo:Fourierinvariance}, we know that there exists an orthogonal projection $q\in \vn(\Gamma)$ such that  $W = q L^2(\vn(\Gamma))$ if and only if $W = \F_\Gamma V$ for some left-invariant  subspace $V \subset \ell_2(\Gamma)$. On the other hand, by (\ref{eq:Fourierintertwining}) we have that $\F_\Gamma \lambda(\gamma) f = (\F_\Gamma f)\rho(\gamma)^*$ for all $f \in \ell_2(\Gamma)$ and all $\gamma \in \Gamma$. Thus, for all $\gamma \in \Gamma$, we have that $\lambda(\gamma) v \in V$ if and only if $(\F_\Gamma v)\rho(\gamma)^* \in W$ for all $v\in V$.
\end{proof}

We now prove that every closed subspace of $\ell_2(\Gamma)$ which is invariant under the left regular representation is principal, and it can be generated by a Parseval frame gnerator.

\begin{prop}\label{prop:Parsevalframeleft}
Every left-invariant closed subspace $V \subset \ell_2(\Gamma)$ is principal, i.e. there exists $\psi \in \ell_2(\Gamma)$ such that
$$
V = \ol{\vsp\{\lambda(\gamma)\psi\}_{\gamma \in \Gamma}}^{\ell_2(\Gamma)}.
$$
Moreover, for $p = \wh{\Proj_V} \in \ell_2(\Gamma)$, the system $\{\lambda(\gamma) p \}_{\gamma \in \Gamma}$ is a Parseval frame for $V$.
\end{prop}
\begin{proof}
Let $V \subset \ell_2(\Gamma)$ be left-invariant. Then, for $f \in V$, using (\ref{eq:Risconvolution})
$$
f = \Proj_V f = f \ast p = \sum_{\gamma \in \Gamma} f(\gamma) \lambda(\gamma)p \in \ol{\vsp\{\lambda(\gamma) p \}_{\gamma \in \Gamma}}^{\ell_2(\Gamma)} ,
$$
which proves that $V \subset \ol{\vsp\{\lambda(\gamma) p \}_{\gamma \in \Gamma}}^{\ell_2(\Gamma)}$. Now, observe that $\Proj_V \in \Proj_V L^2(\vn(\Gamma))$ which coincides with $\F_\Gamma V$ by Theorem \ref{theo:Fourierinvariance}. Then,  $p \in V$ and thus $\ol{\vsp\{\lambda(\gamma) p \}_{\gamma \in \Gamma}}\subset V$, proving the other inclusion.
Then,  we can choose $\psi = p$. 

 Let us see now that the system $\{\lambda(\gamma) p \}_{\gamma \in \Gamma}$ is a Parseval frame for $V$.
For this, note that  by (\ref{eq:bracketleft}) in Proposition \ref{prop:Helsonleft}, the bracket map for $\lambda$ is given by $[f,g] = (\F_\Gamma g)^* (\F_\Gamma f), f,g, \in \ell_2(\Gamma)$. Then, since $\F_\Gamma p =  \Proj_V$, one has $[p,p] = \Proj_V^* \Proj_V = \Proj_V$. So, by \cite[Th. A]{BHP15a}, the system $\{\lambda(\gamma)p\}_{\gamma \in \Gamma}$ is a Parseval frame.
\end{proof}

\section{Invariant subspaces of unitary representations}\label{sec:INVARIANT}

The following result gives a characterization of invariant subspaces in terms in the invariance of its image under a Helson map.
 
\begin{theo}\label{theo:multiplicatively}
Let $(\Gamma,\Pi,\Hil)$ be a dual integrable triple with associated Helson map $\iso$, and let $V \subset \Hil$ be a closed subspace. Then, the following are equivalent
\begin{itemize}
\item[i)] $V$ is $\Pi$-invariant
\item[ii)] $\iso[V]\rho(\gamma) \subset \iso[V]$ for all $\gamma \in \Gamma$
\item[iii)] $\iso[V]F \subset \iso[V]$ for all $F \in \vn(\Gamma)$
\end{itemize}
\end{theo}
\begin{proof}
The equivalence of $i)$ and $ii)$ is a direct consequence of the definition of Helson map, while $iii) \Rightarrow ii)$ is trivial.
We only need to prove $i) \Rightarrow iii)$ To see this, let us first see that
\begin{equation}\label{eq:crucial}
\iso\Big[S_\psi^{-1}(s_{[\psi,\psi]} F)\Big] = \iso[\psi] F
\end{equation}
for every $\psi \in \Hil$ and every $F \in \vn(\Gamma)$, where $S_\psi$ is the isometry given by Proposition \ref{prop:isometry}. To see this, observe first that (\ref{eq:crucial}) holds for trigonometric polynomials as a consequence of (\ref{eq:Helsonproperty}). Let then $F \in \vn(\Gamma)$ and let $\{F_n\}_{n \in \N}$ be a sequence of trigonometric polynomials such that $\{F_n^*\}_{n \in \N}$ converges strongly to $F^*$, i.e.
$$
\|F_n^* u - F^* u\|_{\ell_2(\Gamma)} \to 0, \quad \forall \, u \in \ell_2(\Gamma).
$$
Observe that such a sequence always exists because $\vn(\Gamma)$ coincides with the SOT-closure of trigonometric polynomials by von Neumann's Double Commutant Theorem (see e.g. \cite{Con00}). This implies that for all $\psi \in \Hil$
\begin{equation}\label{eq:convergence}
\|F_n - F\|_{2,[\psi,\psi]} \to 0 .
\end{equation}
Indeed, by definition of the weighted norm we have
\begin{align*}
\|F_n - F\|_{2,[\psi,\psi]}^2 & = \|[\psi,\psi]^\frac12 (F_n - F)\|_2^2  = \tau((F_n - F) (F_n - F)^* [\psi,\psi])\\
& = \langle (F_n - F)^* [\psi,\psi] \delta_\id, (F_n - F)^* \delta_\id\rangle_{\ell_2(\Gamma)}\\
& \leq \|(F_n - F)^* [\psi,\psi] \delta_\id\|_{\ell_2(\Gamma)} \|(F_n - F)^* \delta_\id\|_{\ell_2(\Gamma)}
\end{align*}
where $[\psi,\psi] \delta_\id \in \ell_2(\Gamma)$ because the domain of $[\psi,\psi] \in L^1(\vn(\Gamma)$ contains finite sequences (see e.g. \cite[Section 2]{BHP15a}). Then (\ref{eq:convergence}) follows because $\{F_n^*\}_{n \in \N}$ converges strongly to $F^*$. Now, by Proposition \ref{prop:isometry}, we have
\begin{equation}\label{eq:dummyisometry}
\|S_\psi^{-1}(s_{[\psi,\psi]} F) - S_\psi^{-1}(s_{[\psi,\psi]} F_n) \|_\Hil = \| F - F_n\|_{2,[\psi,\psi]} 
\end{equation}
for all $\psi\in\Hil$ and thus (\ref{eq:convergence}) implies that $S_\psi^{-1}(s_{[\psi,\psi]} F_n)$ converges to $S_\psi^{-1}(s_{[\psi,\psi]} F)$ in $\Hil$. As a consequence, since $\iso$ is continuous, we obtain 
$$
\norm{\iso[S_\psi^{-1}(s_{[\psi,\psi]} F_n)] - \iso[S_\psi^{-1}(s_{[\psi,\psi]} F)]} \to 0 \quad \forall \, \psi \in \Hil.
$$
Since $\iso[S_\psi^{-1}(s_{[\psi,\psi]} F_n)] = \iso[\psi]F_n$, the identity (\ref{eq:crucial}) is proved by showing that $\iso[\psi]F_n$ converges to $\iso[\psi]F$ in $L^2\big((\M,\nu),L^2(\vn(\Gamma))\big)$. Now we have
\begin{align}\label{eq:weightedconvergence}
& \norm{\iso[\psi]F - \iso[\psi]F_n}^2 = \tau \left(\int_\M |\isom{\psi}(F - F_n)|^2 d\nu(x)\right)\nonumber\\
& = \tau \left(|(F - F_n)^*|^2\int_\M |\isom{\psi}|^2 d\nu(x)\right) = \|F - F_n\|^2_{2,[\psi,\psi]} \, ,
\end{align}
where the last identity is due to Proposition \ref{prop:propertyii}. Therefore convergence is provided by (\ref{eq:convergence}). 

Assume that $V$ is $\Pi$-invariant, and take $\psi \in V$ and $F \in \vn(\Gamma)$. Then, by (\ref{eq:crucial}) and Proposition \ref{prop:isometry}, we have 
\begin{displaymath}
\iso[\psi] F = \iso[S_\psi^{-1}(s_{[\psi,\psi]} F)] \in \iso[\psis{\psi}] \subset \iso[V]. \qedhere
\end{displaymath}
\end{proof}

We observe that  a subspace $M$ of $L^2\big((\M,\nu),L^2(\vn(\Gamma))\big)$ satisfying  condition $iii)$ in Theorem \ref{theo:multiplicatively} is what in the abelian case is called {\it multiplicatively invariant} space (see e.g. \cite{BR15}). Then, Theorem \ref{theo:multiplicatively} is a version of \cite[Theorem 3.8]{BR15} in the noncommutative setting, for a discrete  group and general representations.

The next corollary follows directly from the properties of a Helson map.

\begin{cor}
Let $(\Gamma,\Pi,\Hil)$ be a dual integrable triple with associated Helson map $\iso$, and let $V \subset \Hil$ be a $\Pi$-invariant subspace generated by $\{\psi_j\}_{j \in \ind} \subset \Hil$, that is
$$
V = \ol{\vsp\{\Pi(\gamma)\psi_j : j \in \ind, \gamma \in \Gamma}\}^\Hil.
$$
Then
$$
\iso[V] = \ol{\vsp\{\iso[\psi_j]\rho(\gamma) : j \in \ind, \gamma \in \Gamma\}}^{L^2\big((\M,\nu),L^2(\vn(\Gamma))\big)}.
$$
\end{cor}

The following result gives a characterization of the elements belonging to $\psis{\psi}$ in terms of a multiplier that belongs to $L^2(\vn(\Gamma),[\psi,\psi])$.
This extends to the noncommutative setting  \cite[Th. 2.14]{BDR94b}, that is one of the fundamental results in the theory of shift-invariant spaces.
\begin{prop}\label{prop:BDR}
Let $(\Gamma,\Pi,\Hil)$ be a dual integrable triple with  associated Helson map $\iso$ and let $\psi \in \Hil$. Then the following hold:
\begin{itemize}
\item[i)] the mapping $F \mapsto \iso[\psi]F$ from $\hh([\psi,\psi])$ to $L^2\big((\M,\nu),L^2(\vn(\Gamma))\big)$ can be extended by density to an isometry on the whole $L^2(\vn(\Gamma),[\psi,\psi])$;
\item[ii)] $\varphi \in \psis{\psi}$ if and only if there exists $F \in L^2(\vn(\Gamma),[\psi,\psi])$ satisfying
$$
\iso[\varphi] = \iso[\psi]F
$$
and in this case one has $[\varphi,\psi] = [\psi,\psi] F$.
\end{itemize}
\end{prop}

\begin{proof}
In order to see $i)$, it is enough to note that, by (\ref{eq:weightedconvergence}), we have that 
$\norm{\iso[\psi] F}^2 = \|F\|_{2,[\psi,\psi]}^2$ for all $F \in \hh([\psi,\psi])$. Therefore, the conclusion follows.

Let us then prove $ii)$. Observe first that what we have just proved allows us to extend (\ref{eq:crucial}) to
\begin{equation}\label{eq:crucialweight}
\iso[S_\psi^{-1} F] = \iso[\psi] F \quad \forall \ \psi \in \Hil \, , \ F \in L^2(\vn(\Gamma),[\psi,\psi]) .
\end{equation}
Indeed, for $\{F_n\}_{n \in \N} \subset \hh([\psi,\psi])$ a sequence converging to $F \in L^2(\vn(\Gamma),[\psi,\psi])$, we know by (\ref{eq:crucial}) that
$$
\iso[S_\psi^{-1} F_n] = \iso[\psi] F_n \quad \forall \ n \in \N
$$
and, by (\ref{eq:weightedconvergence}), we have that the right hand side converges to $\iso[\psi] F$. By the continuity of $\iso$, in order to show (\ref{eq:crucialweight}) we then need only to prove that $\{S_\psi^{-1} F_n\}_{n \in \N}$ converges to $S_\psi^{-1} F$ in $\Hil$, which is true by (\ref{eq:dummyisometry}).

Now, by Proposition \ref{prop:isometry}, we have that (\ref{eq:crucialweight}) implies that $\varphi \in \psis{\psi}$ if and only if there exists $F \in L^2(\vn(\Gamma),[\psi,\psi])$ satisfying $\iso[\varphi] = \iso[\psi]F$.

As a consequence, by using (\ref{eq:Helsonbracket}), we have that
\[
[\varphi,\psi] = \int_\M \iso[\psi](x)^*\iso[\varphi](x) dx  = \Big(\int_\M \iso[\psi](x)^*\iso[\psi](x) dx \Big) F = [\psi,\psi] F. \qedhere
\]
\end{proof}

Proposition \ref{prop:BDR} extends to finitely generated invariant spaces as follows, generalizing \cite[Theorem 1.7]{BDR94a}
\begin{cor}\label{cor:finitely}
Let $(\Gamma,\Pi,\Hil)$ be a dual integrable triple with associated Helson map $\iso$, and let $V \subset \Hil$ be a $\Pi$-invariant subspace generated by the finite family $\{\psi_j\}_{j = 1}^k \subset \Hil$, that is
$$
V = \ol{\vsp\{\Pi(\gamma)\psi_j : j \in \{1,\dots,k\}, \gamma \in \Gamma\}}^\Hil.
$$
If, for each $j \in \{1,\dots,k\}$, there exists $F_j \in L^2(\vn(\Gamma),[\psi_j,\psi_j])$ such that
\begin{equation}\label{eq:sumclosed}
\iso[\varphi] = \sum_{j=1}^k \iso[\psi_j]F_j ,
\end{equation}
then $\varphi \in V$. Conversely, if $\displaystyle\sum_{j = 1}^k\langle \psi_j \rangle_\Gamma$ is closed and $\varphi \in V$, then there exists $F_j \in L^2(\vn(\Gamma),[\psi_j,\psi_j])$ such that (\ref{eq:sumclosed}) holds.
\end{cor}

\begin{proof}
Assume first that (\ref{eq:sumclosed}) holds. Then, by Proposition (\ref{prop:BDR}), $\iso^{-1}[\iso[\psi_j]F_j] \in \langle \psi_j \rangle_\Gamma$ for all $j=1,\dots,k$, so $\varphi \in \displaystyle\sum_{j = 1}^k \langle \psi_j \rangle_\Gamma \subset V$.

Conversely, if $\displaystyle\sum_{j = 1}^k \langle \psi_j \rangle_\Gamma$ is closed, we have that $\displaystyle\sum_{j = 1}^k \langle \psi_j \rangle_\Gamma =V$. Then, $\varphi \in V$ implies that $\varphi = \displaystyle\sum_{j=1}^k \varphi_j$, where $\varphi_j \in \langle \psi_j\rangle_\Gamma$ for all $j = 1, \dots, k$. So, again the conclusion follows by Proposition (\ref{prop:BDR}).
\end{proof}

Recall that conditions for a sum of subspaces of a Hilbert space to be closed can be found in \cite{Deu95}.

\subsection{Minimality and biorthogonal systems}

In this section we characterize minimal systems, or equivalently biorthogonal systems, in terms of a condition on the bracket map. We recall that, for $\psi \in \Hil$, the system $\{\Pi(\gamma)\psi\}_{\gamma \in \Gamma}$ is said to be minimal if, for all $\gamma_0 \in \Gamma$, it holds
$$
\Pi(\gamma_0)\psi \notin \ol{\vsp\{\Pi(\gamma)\psi \, : \, \gamma \in \Gamma , \gamma \neq \gamma_0\}}^\Hil .
$$
Note that, by the same argument provided in \cite{HSWW10b}, it can be proved that $\{\Pi(\gamma)\psi\}_{\gamma \in \Gamma}$ is minimal if and only if
$$
\psi \notin \ol{\vsp\{\Pi(\gamma)\psi \, : \, \gamma \in \Gamma , \gamma \neq \id\}}^\Hil .
$$
\begin{prop}\label{prop:minimal}
Let $(\Gamma,\Pi,\Hil)$ be a dual integrable triple, and let $0 \neq \psi \in \Hil$. The following are equivalent.
\begin{itemize}
\item[i)] $\{\Pi(\gamma)\psi\}_{\gamma \in \Gamma}$ is minimal.
\item[ii)] There exists $\wt{\psi} \in \psis{\psi}$ such that $\{\Pi(\gamma)\psi\}_{\gamma \in \Gamma}$ and $\{\Pi(\gamma)\wt{\psi}\}_{\gamma \in \Gamma}$ are biorthogonal systems
\item[iii)] $[\psi,\psi]$ is invertible in $\ell_2(\Gamma)$ and $[\psi,\psi]^{-1} \in L^1(\vn(\Gamma))$.
\end{itemize}
\end{prop}
\newpage
\begin{proof}
Recall that $\{\Pi(\gamma)\psi\}_{\gamma \in \Gamma}$ and $\{\Pi(\gamma)\wt{\psi}\}_{\gamma \in \Gamma}$ are biorthogonal systems if
$$
\langle \Pi(\gamma)\psi, \Pi(\gamma')\wt{\psi}\rangle_{\Hil} = \delta_{\gamma,\gamma'} \quad \forall \ \gamma, \gamma' \in \Gamma.
$$
The equivalence of $i)$ and $ii)$ can be carried out following the same argument provided in \cite[Th. 6.1]{HSWW10a}.

Let us prove $ii) \Rightarrow iii)$ Since $\wt{\psi} \in \psis{\psi}$, by Proposition \ref{prop:BDR} there exists $F = F_{\wt{\psi}} \in L^2(\vn(\Gamma),[\psi,\psi])$ such that $\iso[\wt{\psi}] = \iso[\psi]F$ and $[\wt{\psi},\psi] = [\psi,\psi]F$. Moreover, using the definition of dual integrability, it follows that $\{\Pi(\gamma)\psi\}_{\gamma \in \Gamma}$ and $\{\Pi(\gamma)\wt{\psi}\}_{\gamma \in \Gamma}$ are biorthogonal if and only if $[\wt{\psi},\psi] = \Id_{\ell_2(\Gamma)}$. Thus $[\psi,\psi]F = \Id_{\ell_2(\Gamma)}$, which shows that $[\psi,\psi]$ is invertible. Its inverse belongs to $L^1(\vn(\Gamma))$ because
\begin{align*}
\|[\psi,\psi]^{-1}\|_1 & = \tau([\psi,\psi]^{-1}) = \tau(F) = \tau(F[\psi,\psi]F) = \|F\|_{2,[\psi,\psi]} = \|\iso[\psi]F\|_\oplus\\
& = \|\iso[\wt{\psi}]\|_\oplus = \|\wt{\psi}\|_{\Hil} .
\end{align*}

Let us now prove $iii) \Rightarrow ii)$ Since $[\psi,\psi]^{-1} \in L^1(\vn(\Gamma))$, it follows that $[\psi,\psi]^{-1} \in L^2(\vn(\Gamma),[\psi,\psi])$. In fact
$$
\tau(|[\psi,\psi]^{-1}|^2[\psi,\psi]) = \tau([\psi,\psi]^{-1}) = \|[\psi,\psi]^{-1}\|_1 < \infty.
$$
Then, by Proposition \ref{prop:isometry}, there exists $\wt{\psi} \in \psis{\psi}$ such that $S_\psi[\wt{\psi}] = [\psi,\psi]^{-1}$. Since $S_\psi$ is an isometry, for all $\gamma \in \Gamma$ we have
\begin{align*}
\langle \Pi(\gamma)\psi,\wt{\psi}\rangle_{\Hil} & = \langle S_\psi[\Pi(\gamma)\psi],S_\psi[\wt{\psi}]\rangle_{2,[\psi,\psi]} = \langle \rho(\gamma)^*,[\psi,\psi]^{-1}\rangle_{2,[\psi,\psi]}\\
& = \tau(\rho(\gamma)^*[\psi,\psi]^{-1}[\psi,\psi]) = \tau(\rho(\gamma)^*) = \delta_{\gamma,0}
\end{align*}
which shows biorthogonality.
\end{proof}

\section{Frames of orbits}\label{sec:REPRODUCING}

In this section we study reproducing properties of systems of the form 
\begin{equation}\label{eq:systemsoftheform}
E = \{\Pi(\gamma)\phi_i : \gamma \in \Gamma, i \in \ind\}
\end{equation}
where $\{\phi_i\}_{i \in \ind} \subset \Hil$ is a countable family, $(\Gamma,\Pi,\Hil)$ is a dual integrable triple, and $\ind$ is a countable index set. We first show existence of Parseval frames sequences of that form, and then we characterize families $\{\phi_i\}_{i \in \ind}$ for which the system $E$ of their orbits is a Riesz or a frame sequence.

\subsection{Existence of Parseval frames}

The purpose of this subsection is to prove that every $\Pi$-invariant space has a Parseval frame of orbits. We start by doing so for principal spaces, extending \cite[Th. 2.21]{BDR94a} and \cite[Cor. 3.8]{KT08}.

\begin{theo}\label{theo:Parsevalframeprincipal}
Let $(\Gamma,\Pi,\Hil)$ be a dual integrable triple, and let $0 \neq \psi \in \Hil$. Then there exists $\phi \in \Hil$ such that $\{\Pi(\gamma)\phi\}_{\gamma \in \Gamma}$ is a Parseval frame for $\psis{\psi}$.
\end{theo}
\begin{proof}
Let $p = \wh{s_{[\psi,\psi]}} \in \ell_2(\Gamma)$, that is $p(\gamma) = \tau(s_{[\psi,\psi]}\rho(\gamma))$ for every $\gamma \in \Gamma$, and observe that
$$
\begin{array}{rccl}
H_\psi : & \psis{\psi} & \to & \ol{\vsp\{\lambda(\gamma)p\}_{\gamma \in \Gamma}}^{\ell_2(\Gamma)}\vspace{.5ex}\\
& \varphi & \mapsto & \Big\{\tau\Big([\psi,\psi]^\frac12 S_\psi[\varphi]\rho(\gamma)\Big)\Big\}_{\gamma \in \Gamma}
\end{array}
$$
is an isometric isomorphism of Hilbert spaces satisfying
\begin{equation}\label{eq:intertwiningH}
H_\psi[\Pi(\gamma)\varphi] = \lambda(\gamma) H_\psi[\varphi] \quad \forall \ \gamma \in \Gamma, \varphi \in \psis{\psi}.
\end{equation}
Indeed, $[\psi,\psi]^\frac12 S_\psi : \psis{\psi} \to s_{[\psi,\psi]}L^2(\vn(\Gamma))$ is an isometric isomorphism
by Propositions \ref{prop:isometry} and \ref{prop:weightmap}. Now, by Theorem \ref{theo:Fourierinvariance} we know that $V = \big(s_{[\psi,\psi]}L^2(\vn(\Gamma))\big)^\wedge$ is a left-invariant subspace of $\ell_2(\Gamma)$ such that $\Proj_V = s_{[\psi,\psi]} = \F_\Gamma p$, and, by Proposition \ref{prop:Parsevalframeleft}, we have that $V = \ol{\vsp\{\lambda(\gamma)\wh{\Proj_V}\}_{\gamma \in \Gamma}}^{\ell_2(\Gamma)}$. This implies that 
$$ H_\psi:\psis{\psi} \to \ol{\vsp\{\lambda(\gamma)p\}_{\gamma \in \Gamma}}^{\ell_2(\Gamma)}$$ is an isometric isomorphism.
Additionally, by (\ref{eq:intertwining}) it follows that $$[\psi,\psi]^\frac12 S_\psi[\Pi(\gamma)\varphi] = [\psi,\psi]^\frac12 S_\psi[\varphi]\rho(\gamma)^* \quad \forall \ \gamma \in \Gamma, \varphi \in \psis{\psi}.$$ Thus, for $\gamma, \gamma' \in \Gamma$, we have
\begin{align*}
H_\psi[\Pi(\gamma)\varphi](\gamma') & = \tau\Big([\psi,\psi]^\frac12 S_\psi[\Pi(\gamma)\varphi]\rho(\gamma')\Big) = \tau\Big([\psi,\psi]^\frac12 S_\psi[\varphi]\rho(\gamma)^*\rho(\gamma')\Big)\\
& = \tau\Big([\psi,\psi]^\frac12 S_\psi[\varphi]\rho(\gamma^{-1}\gamma')\Big) = H_\psi[\varphi](\gamma^{-1}\gamma') = \lambda(\gamma) H_\psi[\varphi](\gamma')
\end{align*}
hence proving (\ref{eq:intertwiningH}).

Let now $\phi = H_\psi^{-1}[p]$. Then, for $\varphi \in \psis{\psi}$, since $\{\lambda(\gamma)p\}_{\gamma \in \Gamma}$ is a Parseval frame sequence by Proposition \ref{prop:Parsevalframeleft}, we have
\begin{align*}
\sum_{\gamma \in \Gamma} |\langle \varphi, \Pi(\gamma)\phi\rangle_\Hil|^2 & = \sum_{\gamma \in \Gamma} |\langle H_\psi[\varphi], H_\psi[\Pi(\gamma)\phi]\rangle_{\ell_2(\Gamma)}|^2 = \sum_{\gamma \in \Gamma} |\langle H_\psi[\varphi], \lambda(\gamma) p\rangle_{\ell_2(\Gamma)}|^2\\
& = \|H_\psi[\varphi]\|_{\ell_2(\Gamma)}^2 = \|\varphi\|_\Hil^2 \ ,
\end{align*}
showing that $\{\Pi(\gamma)\phi\}_{\gamma \in \Gamma}$ is a Parseval frame for $\psis{\psi}$.
\end{proof}

\begin{cor}\label{cor:Parsevalframe}
Let $V \subset \Hil$ be a $\Pi$-invariant subspace. Then there exist a countable family $\{\phi_i\}_{i \in \ind} \subset \Hil$ such that $E = \{\Pi(\gamma)\phi_i : \gamma \in \Gamma, i \in \ind\}$ is a Parseval frame for $V$.
\end{cor}

\begin{proof}
Consider a family $\{\psi_i\}_{i \in \ind}$ as in Lemma \ref{lem:generators}. Now, for each $i \in \ind$, let $\phi_i$ be the Parseval frame generator of $\psis{\psi_i}$ given by Theorem \ref{theo:Parsevalframeprincipal}. Since $\psis{\phi_i} \perp \psis{\phi_j}$ for $i \neq j$, the system $E$ is a Parseval frame for $V$.
\end{proof}

We remark that this corollary extends to general discrete groups and unitary representations the following results \cite[Th. 3.3]{Bow00}, \cite[Th. 3.10]{KT08}, \cite[Th. 4.11]{CP10}, \cite[Th. 5.5]{BHP15b} (see also \cite[Th. 5.3]{BR15}).

\subsection{Characterization of frames and Riesz systems}

This subsection is devoted to characterize the reproducing properties of systems of the form (\ref{eq:systemsoftheform}).

For instance, we can easily see that $E$ is an orthonormal system if and only if
\begin{equation}\label{eq:onb}
[\phi_i, \phi_j] = \delta_{i,j} \Id_{\ell_2(\Gamma)}.
\end{equation}
Indeed, observe first that by definition of the bracket map we have that, for $i \neq j$
$$
\psis{\phi_i} \bot \psis{\phi_j} \iff [\phi_i,\phi_j] = 0
$$
because
$$
[\phi_i,\phi_j] = 0 \iff 0 = \tau([\phi_i,\phi_j]\rho(\gamma)) = \langle \phi_i, \Pi(\gamma)\phi_j\rangle_\Hil \quad \forall \ \gamma \in \Gamma.
$$
Moreover, for each $i \in \ind$, we have that $\{\Pi(\gamma)\phi_i\}_{\gamma \in \Gamma}$ is an orthonormal system if and only if $[\phi_i,\phi_i] = \Id_{\ell_2(\Gamma)}$ by the same argument as above (see also \cite[i), Th. A]{BHP15a}).

For the case of Riesz and frame sequences, the characterization is not as simple, and it will be the content of the next two theorems. The structure of their proofs is analogous to the one developed for the abelian cases in \cite[Th. 2.3]{Bow00} and \cite[Th. 4.1 and Th. 4.3]{CP10}.

\begin{theo}\label{th:riesz}
Let $(\Gamma,\Pi,\Hil)$ be a dual integrable triple, let $\{\phi_i\}_{i \in \ind} \subset \Hil$ be a countable family, and denote by $E$ the system
$$
E = \{\Pi(\gamma)\phi_i : \gamma \in \Gamma, i \in \ind\}.
$$
Given two constants $0 < A \leq B < \infty$, the following conditions are equivalent:
\begin{itemize}
\item[i)] $E$ is a Riesz sequence with frame bounds $A , B$.
\item[ii)] $\displaystyle A \sum_{i \in \ind} |F_i|^2 \leq \sum_{i,j \in \ind} F_j^*[\phi_i,\phi_j]F_i \leq B \sum_{i \in \ind} |F_i|^2$\\
for all finite sequence $\{F_i\}_{i \in \ind}$ in $\vn(\Gamma)$.
\end{itemize}
\end{theo}

\begin{proof}
Note first that, if $\iso$ is a Helson map associated to $(\Gamma,\Pi,\Hil)$, by Proposition \ref{prop:propertyii} we have
$$
\sum_{i,j \in \ind} F_j^*[\phi_i,\phi_j]F_i = \int_\M \Big| \sum_{i \in \ind} \isom{\phi_i}F_i\Big|^2d\nu(x) .
$$

Let $\{b(\gamma,i) : \gamma \in \Gamma, i \in \ind\}$ be a finite sequence. Then, by the properties of the Helson map, we have
\begin{align*}
\bigg\|\sum_{{\gamma \in \Gamma} \atop {i \in \ind}} b(\gamma, i) \Pi(\gamma)\phi_i\bigg\|_\Hil^2
& = \bigg\|\iso\Big[\sum_{{\gamma \in \Gamma} \atop {i \in \ind}} b(\gamma, i) \Pi(\gamma)\phi_i\Big]\bigg\|_{\oplus}^2\\
& = \int_\M \tau\bigg(\Big|\sum_{i \in \ind}\isom{\phi_i}\sum_{\gamma \in \Gamma} b(\gamma, i) \rho(\gamma)^*\Big|^2\bigg) d\nu(x).
\end{align*}
On the other hand, if we call $F_i = \displaystyle\sum_{\gamma \in \Gamma} b(\gamma, i) \rho(\gamma)^*$, by Plancherel Theorem we have
$$
\sum_{{\gamma \in \Gamma} \atop {i \in \ind}} |b(\gamma, i)|^2 = \sum_{i \in \ind} \tau(|F_i|^2) .
$$
Then, condition $i)$ of $E$ being a Riesz sequence is equivalent to the condition\vspace{4pt}\\
$iii) \quad \displaystyle A \tau\bigg(\sum_{i \in \ind} |F_i|^2\bigg) \leq \tau \bigg( \int_\M \Big|\sum_{i \in \ind}\isom{\phi_i}F_i\Big|^2 d\nu(x)\bigg) \leq B \tau\bigg(\sum_{i \in \ind} |F_i|^2\bigg)$. \vspace{4pt}\\
for all finite sequence $\{F_i\}_{i \in \ind} \subset \vn(\Gamma)$.

We then prove the equivalence of $ii)$ and $iii)$. The implication $ii) \Rightarrow iii)$ is trivial, since for all positive operators $P$ on $\ell^2(\Gamma)$ one has $\tau(P) \geq 0$.

In order to prove $iii) \Rightarrow ii)$ we proceed by contradiction.
Suppose indeed that the right inequality in $ii)$ does not hold for a finite sequence $\{F_i\}_{i \in \ind} \subset \vn(\Gamma)$, and define $\Proj$ to be the orthogonal projection
$$
\Proj = \chi_{{}_{(0,\infty)}}\Big(\int_\M\Big|\sum_{i \in \ind} \isom{\phi_i}F_i\Big|^2d\nu(x) - B \sum_{i \in \ind} |F_i|^2\Big)
$$
where $\chi_{{}_\Omega}(F)$ stands for the spectral projection of the selfadjoint operator $F$ over the Borel set $\Omega \subset \R$. By \cite[Theorem 5.3.4]{Sun97}, since we are defining a spectral projection of a closed and densely defined selfadjoint affiliated operator, then $\Proj \in \vn(\Gamma)$. Then $W = \Ran(\Proj)$ is the closed linear subspace of $\ell_2(\Gamma)$ where the right inequality in $ii)$ does not hold, and
$$
\langle \Big(\int_\M\Big|\sum_{i \in \ind} \isom{\phi_i}F_i\Big|^2d\nu(x) - B \sum_{i \in \ind} |F_i|^2\Big) u, u\rangle_{\ell^2(\Gamma)} > 0 \quad \forall \ u \in W .
$$
This means that
\begin{equation}\label{eq:dummy28}
\Proj \Big(\displaystyle{\int_\M}\Big|\sum_{i \in \ind} \isom{\phi_i}F_i\Big|^2d\nu(x) - B \sum_{i \in \ind} |F_i|^2\Big) \Proj > 0 . 
\end{equation}
We now write
\begin{align*}
\Proj \Big(\int_\M\Big|\sum_{i \in \ind} & \isom{\phi_i}F_i\Big|^2d\nu(x) - B \sum_{i \in \ind} |F_i|^2\Big) \Proj\\
& = \int_\M\sum_{i, j \in \ind} \Proj F_j^*\isom{\phi_j}^*\isom{\phi_i}F_i \Proj d\nu(x)  - B \sum_{i,j \in \ind} \Proj F_j^*F_i \Proj\\
& = \int_\M\sum_{i,j \in \ind} {F^W_j}^*\isom{\phi_j}^*\isom{\phi_i}F^W_i d\nu(x)  - B \sum_{i,j \in \ind} {F^W_j}^*F^W_i 
\end{align*}
where we have used the shorthand notation $F^W_i = F_i \Proj \in \vn(\Gamma)$. By the linearity of $\tau$, we can then deduce from (\ref{eq:dummy28}) that
$$
\tau \bigg( \int_\M \Big|\sum_{i \in \ind}\isom{\phi_i}F^W_i\Big|^2 d\nu(x)\bigg) > B \tau\bigg(\sum_{i \in \ind} |F^W_i|^2\bigg)
$$
which contradicts the right inequality of $iii)$. When the inequality at the left hand side fails, we can proceed analogously and obtain a similar contradiction.
\end{proof}

\begin{rem}
The characterization of orthonormal systems given by (\ref{eq:onb}) can be also deduced from Theorem \ref{th:riesz} as follows. Item $ii)$, Theorem \ref{th:riesz} for orthonormal systems reads
\begin{equation}\label{eq:onbdummy}
\sum_{i,j \in \ind} F_j^*[\phi_i,\phi_j]F_i = \sum_{i \in \ind} |F_i|^2
\end{equation}
for all finite sequence $\{F_i\}_{i \in \ind} \subset \vn(\Gamma)$. If (\ref{eq:onb}) holds, this identity is trivial. Conversely, for each $k \in \ind$ consider the finite sequence $\{\delta_{j,k} \Id_{\ell_2(\Gamma)}\}_{j \in \ind} \subset \vn(\Gamma)$ and apply (\ref{eq:onbdummy}) to obtain $[\phi_k,\phi_k] = \Id_{\ell_2(\Gamma)}$. Using this, and applying (\ref{eq:onbdummy}) to the sequence $\{(\delta_{j,k_1} + \delta_{j,k_2}) \Id_{\ell_2(\Gamma)}\}_{j \in \ind} \subset \vn(\Gamma)$ with $k_1 \neq k_2$ we then get
$$
2 \Id_{\ell_2(\Gamma)} + [\phi_{k_1},\phi_{k_2}] + [\phi_{k_2},\phi_{k_1}] = 2 \Id_{\ell_2(\Gamma)} .
$$
Analogously, for the sequence $\{(\delta_{j,k_1} + i \delta_{j,k_2}) \Id_{\ell_2(\Gamma)}\}_{j \in \ind} \subset \vn(\Gamma)$ with $k_1 \neq k_2$ we obtain
$$
2 \Id_{\ell_2(\Gamma)} - i( [\phi_{k_1},\phi_{k_2}] - [\phi_{k_2},\phi_{k_1}]) = 2 \Id_{\ell_2(\Gamma)} .
$$
Thus $[\phi_{k_1},\phi_{k_2}] = 0$.

\end{rem}

\begin{theo}\label{theo:frames}
Let $(\Gamma,\Pi,\Hil)$ be a dual integrable triple, let $\{\phi_i\}_{i \in \ind} \subset \Hil$ be a countable family, and denote with $E$ the system
$$
E = \{\Pi(\gamma)\phi_i : \gamma \in \Gamma, i \in \ind\}.
$$
Given two constants $0 < A \leq B < \infty$, the following conditions are equivalent:
\begin{itemize}
\item[i)] $E$ is a frame sequence with frame bounds $A , B$.
\item[ii)] $A [f,f] \leq \displaystyle \sum_{i \in \ind} |[f, \phi_i]|^2 \leq B [f,f]$ \ for all $f \in \ol{\vsp\, E}^\Hil$ .
\end{itemize}
\end{theo}

\begin{proof}
The structure of the proof is similar to that of the previous theorem.

By the definition of bracket map and Plancherel Theorem, for all $f \in \ol{\vsp\, E}^\Hil$ we have
$$
\sum_{\gamma \in \Gamma} |\langle f, \Pi(\gamma)\phi_i\rangle_\Hil|^2 = \sum_{\gamma \in \Gamma} |\tau(\rho(\gamma)[f,\phi_i])|^2 = \tau \big(|[f,\phi_i]|^2\big) \quad \forall \ i \in \ind
$$
so that the condition $i)$ of $E$ being a frame system is equivalent to the condition\vspace{4pt}\\
$iii) \quad A \tau([f,f]) \leq \displaystyle \sum_{i \in \ind} \tau \big(|[f,\phi_i]|^2\big) \leq B \tau([f,f])$ \ for all $f \in \ol{\vsp\,E}^\Hil$\vspace{4pt}\\
since by property III) of the bracket map $\tau([f,f]) = \|f\|_\Hil^2$.
We then prove the equivalence of $ii)$ and $iii)$. As for Theorem \ref{th:riesz}, the implication $ii) \Rightarrow iii)$ is trivial. In order to prove that $iii)$ implies $ii)$ we proceed by contradiction.
Suppose indeed that the right inequality in $ii)$ does not hold for some $f_0 \in \ol{\vsp\,E}^\Hil$, and let us define the orthogonal projection of $\vn(\Gamma)$
$$
\Proj = \chi_{{}_{(0,\infty)}}\Big(\sum_{i \in \ind} |[f_0, \phi_i]|^2 - B [f_0,f_0]\Big) .
$$

Let $W = \Ran(\Proj)$, and note that $W$ is the closed linear subspace of $\ell_2(\Gamma)$ where the right inequality in $ii)$ does not hold for $f_0$. Then
$$
\langle \Big(\sum_{i \in \ind} |[f_0, \phi_i]|^2 - B [f_0,f_0]\Big) u, u\rangle_{\ell^2(\Gamma)} > 0 \quad \forall \ u \in W
$$
which means that
$$
0 < \Proj \Big(\sum_{i \in \ind} |[f_0, \phi_i]|^2 - B [f_0,f_0]\Big) \Proj = \sum_{i \in \ind} \Proj [\phi_i , f_0][f_0,\phi_i]\Proj  - B \Proj [f_0,f_0]\Proj.
$$
Now, by $iii)$, Theorem \ref{theo:multiplicatively}, we have that since $f_0 \in \ol{\vsp\, E}^\Hil$, there exists $f_W \in \ol{\vsp\, E}^\Hil$ such that $\iso[f_0]\Proj = \iso[f_W]$. So, by Proposition \ref{prop:propertyii}
$$
[f_0,\phi_i]\Proj = \int_{\M} \iso[\phi_i](x)^* \iso[f_0](x) \Proj d\nu(x) = [f_W,\phi_i].
$$
Proceeding analogously for the other brackets, we then get
$$
0 < \sum_{i \in \ind} |[f_W, \phi_i]|^2  - B [f_W, f_W].
$$
By the linearity of $\tau$ we could then deduce that
$$
\tau\Big(\sum_{i \in \ind} |[f_W, \phi_i]|^2\Big) > B \tau([f_W,f_W])
$$
which contradicts the right inequality of $iii)$.
\end{proof}

In the case of only one generator, we can recover \cite[Th. A]{BHP15a} as a corollary. We emphasize that this type of result was first proved for the case of integer translations in \cite{BL93, BW94}.
\begin{cor}\label{cor:principal}
Let $\phi \in \Hil$, let $E = \{\Pi(\gamma)\phi : \gamma \in \Gamma\}$ and let $0 < A \leq B < \infty$. Then 
\begin{itemize}
\item[i)] $E$ is a Riesz sequence if and only if $A \Id_{\ell_2(\Gamma)} \leq [\phi,\phi] \leq B \Id_{\ell_2(\Gamma)}$;
\item[ii)] $E$ is a frame sequence if and only if $A s_{[\phi,\phi]} \leq [\phi,\phi] \leq B s_{[\phi,\phi]}$.
\end{itemize}

\end{cor}
\begin{proof}
To prove $i)$, note that by $ii)$, Theorem \ref{th:riesz} we have that $E$ is a Riesz sequence if and only if
$$
A |F|^2 \leq F^*[\phi,\phi]F \leq B |F|^2 \quad \forall \ F \in \vn(\Gamma).
$$
which is easily seen to be equivalent to $A \Id_{\ell_2(\Gamma)} \leq [\phi,\phi] \leq B \Id_{\ell_2(\Gamma)}$.

To prove $ii)$, by $ii)$, Theorem \ref{theo:frames} we have that $E$ is a frame sequence if and only if
$$
A [f,f] \leq \displaystyle |[f, \phi]|^2 \leq B [f,f] \quad \forall \ f \in \ol{\vsp\, E}^\Hil = \psis{\phi}.
$$
Now, by $ii)$, Proposition \ref{prop:BDR}, we have that for any $f \in \psis{\phi}$ there exits a unique $F \in L^2(\vn(\Gamma),[\phi,\phi])$ such that $\iso[f] = \iso[\phi] F$ and $[f, \phi] = [\phi, \phi]F$. By Proposition \ref{prop:propertyii}, we also have that
$$
[f,f] = F^*[\phi,\phi]F, 
$$
so, recalling Proposition \ref{prop:isometry}, the previous inequalities read
$$
A F^*[\phi,\phi]F \leq \displaystyle F^*|[\phi, \phi]|^2F \leq B F^*[\phi,\phi]F \quad \forall \ F \in L^2(\vn(\Gamma),[\phi,\phi]) .
$$
This is easily seen to be equivalent to $A s_{[\phi,\phi]} \leq [\phi,\phi] \leq B s_{[\phi,\phi]}$.
\end{proof}

\section{Relevant examples}\label{sec:EXAMPLES}

In this section we provide examples of brackets and Helson maps in different settings.

\subsection{Integer translations on $L^2(\R)$.}\label{sec:translates} 
Let $\Gamma$ be a uniform lattice of an LCA group $G$, i.e. a discrete and countable subgroup such that $G/\Gamma$ is compact, and let $T : \Gamma \to \mathcal{U}(L^2(G))$ be  given by $T(\gamma)f(x) = f(x - \gamma)$. A fundamental tool for analyzing the structure of shift-invariant subspaces is the so-called fiberization mapping (see \cite[Prop. 3.3]{CP10}):
\begin{align*}
\iso : L^2(G) \to L^2(\Omega,\ell_2(\Gamma^\bot))\\
\iso[f](\omega) = \{\F_G f(\omega + \lambda)\}_{\lambda \in \Gamma^\bot}
\end{align*}
where $\Omega$ is a measurable section of the quotient $\wh{G}/\Gamma^\bot$, $\Gamma^\bot$ is the annihilator of $\Gamma$ (which is discrete), $\wh G$ is de dual group of $G$, and $\F_G f(\chi) = \int_G f(x) \ol{\chi(x)} dx$,  for $\chi\in \wh G$, is the Fourier transform in the LCA group $G$. Recall that the annihilator of a group $K\subseteq G$ is the closed subgroup of $\wh G$ given by  $K^\perp=\{\chi\in\wh G: \chi(\kappa)=1 \ \forall \, \kappa \in K\}$.

We want to show that this map can actually be obtained as a special case of the construction given by Proposition \ref{prop:periodization}.

First of all one must take into account that, when $\Gamma$ is abelian, there is an isomorphism between $\vn(\Gamma)$ and $\wh{\Gamma} \approx \wh{G}/\Gamma^\bot \approx \Omega$, provided by Pontryagin duality (see also \cite{BHP17}). Therefore,   the target space of the map $U_\Psi$ of  Proposition \ref{prop:periodization} is
$$
\ell_2(\ind,L^2(\vn(\Gamma))) \approx \ell_2(\ind,L^2(\Omega)) \approx L^2(\Omega,\ell_2(\ind)) .
$$
Now, for the sake of simplicity, we will work in detail the case $G = \R$, $\Gamma = \Z$ and $T$ the integer translations on $L^2(\R)$, i.e. $T(k)\varphi(x) = \varphi(x - k)$ (see \cite{Bow00}).

Let $\ind = \Gamma^\bot = \Z$ be the annihilator of $\Gamma = \Z$. Consider $\Psi = \{ \psi_j\}_{j \in \Z}\subset L^2(\R)$ be the Shannon system
\begin{equation}\label{eq:dummy1}
\F_\R \psi_j = \chi_{[j,j+1]}, \,\, j\in\Z.
\end{equation}
If $ \langle \psi_j \rangle_\Z = \ol{\vsp\{T(k)\psi_j \, : \, k \in \Z\}}^{L^2(\R)}$, it is clear that
$$
L^2(\R) = \bigoplus_{j \in \Z} \langle \psi_j \rangle_\Z
$$
because $\F_\R T(k)\psi_j(\omega) = \chi_{[j,j+1]}e^{-2\pi i k \omega}$. Moreover, the integer translates of each $\psi_j$ generate an orthonormal system, so that $[\psi_j,\psi_j] = \Id_{\ell_2(\Z)}$ (see \cite[Theorem A]{BHP15a}). Then, $\Psi = \{ \psi_j\}_{j \in \Z}\subset L^2(\R)$ is a family as in Lemma \ref{lem:generators} and  the map of Proposition \ref{prop:periodization} is $U_{\Psi}[\varphi] = \{S_{\psi_j}[\Proj_{ \langle \psi_j \rangle_\Z}\varphi]\}_{j \in \Z}$ for $\varphi \in L^2(\R)$. Write
$$
\Proj_{ \langle \psi_j \rangle_\Z}\varphi(x) = \sum_{k \in \Z} a^j_k \psi_j(x - k) = \sum_{k \in \Z} a^j_k T(k)\psi_j(x)
$$
with $a^j_k = \langle \varphi, T(k)\psi_j\rangle_{L^2(\R)}$. 
Then,
\begin{equation}\label{eq:integerisometry}
U_\Psi[\varphi] = \Big\{\sum_{k \in \Z} a_k^j \rho(k)^*\Big\}_{j \in \Z},
\end{equation}
where $\{\rho(k)\}_{k \in \Z}$ is the sequence of translation operators in $\ell^2(\Z)$.

We now show that $U_\Psi$ gives rise to the map $\iso : L^2(\R) \to L^2([0,1],\ell_2(\Z))$ given by
$\iso[f](\omega) = \{\F_\R f(\omega + j)\}_{j \in \Z}$ by replacing the integer translations $\{\rho(k)\}_{k \in \Z}$ of $\ell^2(\Z)$ with the characters $\{e^{2\pi i k \cdot}\}_{k \in \Z}$ of $\Z$. 

By definition, $\F_\R \psi_j (\omega + l) = \delta_{j,l}$ for all $j, l \in \Z$ and a.e. $\omega \in [0,1)$. Thus, for $\varphi \in L^2(\R)$
\begin{align*}
\F_\R \Proj_{ \langle \psi_j \rangle_\Z}\varphi (\omega + l) & = \sum_{k\in\Z} a^j_k \F_\R T(k)\psi_j(\omega + l) = \sum_{k\in\Z} a^j_k \chi_{[j,j+1]}(\omega + l) e^{-2\pi i k\omega}\\
& = \sum_{k\in\Z} a^j_k \delta_{j,l} e^{-2\pi i k\omega} \ , \quad \textnormal{a.e.} \ \omega \in [0,1).
\end{align*}
Then,
$$
\sum_{j \in \Z} \F_\R \Proj_{ \langle \psi_j \rangle_\Z}\varphi (\omega + l) = \sum_{k\in\Z} a^l_k e^{-2\pi i k\omega} \ , \quad \textnormal{a.e.} \ \omega \in [0,1).
$$
Therefore, $U_\Psi$ becomes
$$
\Big\{\sum_{k \in \Z} a_k^j e^{-2\pi i k\omega}\Big\}_{j \in \Z} = \Big\{\sum_{j \in \Z} \F_\R \Proj_{ \langle \psi_j \rangle_\Z}\varphi (\omega + l)\Big\}_{l \in \Z} = \Big\{\F_\R \varphi (\omega + l)\Big\}_{l \in \Z}=\iso[f](\omega),
$$
for a.e. $\omega \in [0,1)$.

For the general case, consider the family $\F_G \psi_\delta = \chi_{\Omega + \delta} \, , \ \delta \in \Gamma^\bot$ instead of (\ref{eq:dummy1}).
The rest of the details are left to the reader.

\subsection{Measurable group actions on $L^2(\meas,\mu)$ and Zak transform.}
A particular construction of a Helson map can be given in terms of the Zak transform whenever the representation $\Pi$ arises from a measurable action of a discrete group on a measure space. This was first considered in the abelian setting in \cite{HSWW10a} and then in \cite{BHP15b}. For the nonconmmutative case, the Zak transform was taken into consideration in \cite{BHP15a}. For the sake of completeness we include its construction here.

Consider a $\sigma$-finite measure space $(\meas,\mu)$, $\Gamma$ a countable discrete group  and let $\sigma: \Gamma  \times \meas \to \meas$ be a quasi $\Gamma$-invariant measurable action of $\Gamma$ on $\meas$. This means that for each $\gamma \in \Gamma$ the map $x \mapsto \sigma_\gamma (x) = \sigma(\gamma,x)$ is $\mu$-measurable, that for all $\gamma, \gamma' \in \Gamma$ and almost all $x \in \meas$ it holds $\sigma_\gamma (\sigma_{\gamma'} (x)) = \sigma_{\gamma \gamma'}(x)$ and $\sigma_\id (x) = x$, and that for each $\gamma \in \Gamma$  the measure $\mu_\gamma$ defined by $\mu_\gamma(E) = \mu(\sigma_\gamma (E))$ is absolutely continuous with respect to $\mu$ with positive Radon-Nikodym derivative. Let us indicate the family of associated Jacobian densities with the measurable function $J_\sigma: \Gamma \times \meas \to \R^+$ given by
$$
d\mu(\sigma_\gamma (x)) = J_\sigma (\gamma, x)\, d\mu (x) .
$$
We can then define a unitary representation $\Pi_\sigma$ of $\Gamma$ on $L^2(\meas,\mu)$ as
\begin{equation}\label{eq:measurablerep}
\Pi_\sigma(\gamma)\varphi(x) = J_\sigma (\gamma^{-1}, x)^\frac12 \varphi(\sigma_{\gamma^{-1}}(x)) . 
\end{equation}
We say that the action $\sigma$ has the tiling property if there exists a $\mu$-measurable subset $C \subset \meas$ such that the family
$\{\sigma_\gamma (C)\}_{\gamma \in \Gamma}$ is a $\mu$-almost disjoint covering of $\meas$, i.e.
$\mu\big(\sigma_{\gamma_1}(C) \cap \sigma_{\gamma_2} (C)\big) = 0$ for $\gamma_1 \neq \gamma_2$ and
$$
\mu\bigg(\meas \setminus \bigcup_{\gamma \in \Gamma}\,\sigma_\gamma (C)\bigg) = 0 .
$$

Following \cite{BHP15a}, the noncommutative Zak transform of $\varphi \in L^2(\meas,\mu)$ associated to the action $\sigma$ is given by
$$
Z_\sigma[\varphi](x) = \sum_{\gamma \in \Gamma} \Big(\big(\Pi_\sigma(\gamma)\varphi\big)(x)\Big) \rho(\gamma), \quad x \in \meas .
$$

The following result is a slight improvement of \cite[i), Th. B]{BHP15a}, showing that $Z_\sigma$ defines an isometry that is surjective on the whole $L^2((C,\mu),L^2(\vn(\Gamma)))$.
\begin{prop}\label{prop:Zak}
Let $\sigma$ be a quasi-$\Gamma$-invariant action of the countable discrete group $\Gamma$ on the measure space $(\meas,\mu)$, and let $\Pi_\sigma$ be the unitary representation given by  (\ref{eq:measurablerep}) on $L^2(\meas,\mu)$. If $\sigma$ has the tiling property with tiling set $C$, then the Zak transform $Z_\sigma$ defines an isometric isomorphism
$$
Z_\sigma : L^2(\meas,\mu) \to L^2((C,\mu),L^2(\vn(\Gamma)))
$$
satisfying the condition
\begin{equation}\label{eq:Helson-Zak}
Z_\sigma[\Pi_\sigma(\gamma)\varphi] = Z_\sigma[\varphi] \rho(\gamma)^* \, , \quad  \forall \, \gamma \in \Gamma , \ \forall \, \varphi \in L^2(\meas,\mu) .
\end{equation}
Hence, $Z_\sigma$ is a Helson map for the representation $\Pi_\sigma$. As a consequence, the bracket map for $\Pi_\sigma$ can be written as
$$
[\varphi, \psi] = \int_{\meas} Z_\sigma[\psi](x)^* Z_\sigma[\varphi](x) d\mu(x).
$$
\end{prop}

\begin{proof}
The isometry can be proved as in \cite[Th. B]{BHP15a}, while property (\ref{eq:Helson-Zak}) can be obtained explicitly by
\begin{align*}
Z_\sigma[\Pi_\sigma(\gamma)\varphi] & = \sum_{\gamma'} \Big(\big(\Pi_\sigma(\gamma'\gamma)\varphi\big)(x)\Big) \rho(\gamma') = \sum_{\gamma''} \Big(\big(\Pi_\sigma(\gamma'')\varphi\big)(x)\Big) \rho(\gamma''\gamma^{-1}) \\
& = Z_\sigma[\varphi] \rho(\gamma)^*.
\end{align*}
To prove surjectivity, take $\varPsi \in L^2((C,\mu),L^2(\vn(\Gamma)))$ and for each $\gamma \in \Gamma$ define
\begin{equation}\label{eq:Zakinversion}
\psi(x) = J_\sigma (\gamma^{-1}, \sigma_{\gamma}(x))^{-\frac12} \tau \big( \varPsi(\sigma_{\gamma}(x)) \rho(\gamma)^* \big) \quad \textnormal{a.e.} \ x \in \sigma_{\gamma^{-1}}(C) .
\end{equation}
Such a $\psi$ belongs to $L^2(\meas,\mu)$, since by the tiling property it is measurable and its norm reads
\begin{align*}
\|\psi\|_{L^2(\meas,\mu)}^2 & = \sum_{\gamma \in \Gamma} \int_{\sigma_{\gamma^{-1}}(C)} J_\sigma (\gamma^{-1}, \sigma_{\gamma}(x))^{-1} \Big|\tau \big( \varPsi(\sigma_{\gamma}(x)) \rho(\gamma)^* \big)\Big|^2 d\mu(x)\\
& = \sum_{\gamma \in \Gamma} \int_{C} J_\sigma (\gamma^{-1}, y)^{-1} |\tau (\varPsi(y) \rho(\gamma)^* )|^2 J_\sigma (\gamma^{-1}, y)d\mu(y)
\end{align*}
where the last identity is due to the definition of the Jacobian density, because 
$d\mu(x)=d\mu(\sigma_{\gamma^{-1}}(y))=J_\sigma (\gamma^{-1}, y)d\mu(y)$.
Then, by Plancherel Theorem
$$
\|\psi\|_{L^2(\meas,\mu)}^2 = \int_C \sum_{\gamma \in \Gamma} |\tau (\varPsi(y) \rho(\gamma)^* )|^2  d\mu(y) = \int_C \|\varPsi(y)\|_2^2  d\mu(y)  
$$
so that $\|\psi\|_{L^2(\meas,\mu)}^2 = \|\varPsi\|^2_{L^2((C,\mu),L^2(\vn(\Gamma)))}<+\infty$. By applying the Zak transform to $\psi$ we then have that, for a.e. $x \in C$,
\begin{align*}
Z_\sigma[\psi](x) = \sum_{\gamma \in \Gamma} J_\sigma (\gamma^{-1}, x)^\frac12 \psi(\sigma_{\gamma^{-1}}(x)) \rho(\gamma) = \sum_{\gamma \in \Gamma} \tau ( \varPsi(x) \rho(\gamma)^* \big) \rho(\gamma) = \varPsi(x)
\end{align*}
again by Plancherel Theorem. This proves surjectivity and in particular shows that (\ref{eq:Zakinversion}) provides an explicit inversion formula for $Z_\sigma$.
\end{proof}

\begin{rem}
The Zak transform is actually directly related to the isometry $S_\psi$ introduced in (\ref{eq:isometry}), since for all $F \in \hh([\psi,\psi])$ (see Section \ref{sec:weight}) it holds
$$
F = S_\psi \Big(\tau \big( Z_\sigma[\psi](\cdot) F \big)\Big).
$$
First notice that $\tau \big( Z_\sigma[\psi](\cdot) F \big) \in \psis{\psi}$. Indeed, let $F \in \vsp\{\rho(\gamma)\}_{\gamma \in \Gamma} \cap \hh([\psi,\psi])$, and denote with $\{\wh{F}(\gamma)\}_{\gamma \in \Gamma}$ its Fourier coefficients. By the orthonormality of $\{\rho(\gamma)\}_{\gamma \in \Gamma}$ in $L^2(\vn(\Gamma))$ it holds
$$
\tau \big( Z_\sigma[\psi](x) F \big) = \sum_{\gamma, \gamma' \in \Gamma} \big(\Pi_\sigma(\gamma) \psi\big) (x) \wh{F}(\gamma') \, \tau(\rho(\gamma)\rho(\gamma')^*) = \sum_{\gamma \in \Gamma} \wh{F}(\gamma) \Pi_\sigma(\gamma)\psi(x)
$$
for a.e. $x \in \meas$.
Therefore $\tau \big( Z_\sigma[\psi](\cdot) F \big) \in \vsp\{\Pi_\sigma(\gamma)\psi\}$. Consequently
$$
S_\psi \Big(\tau \big( Z_\sigma[\psi](\cdot) F \big)\Big) = s_{[\psi,\psi]} \sum_{\gamma \in \Gamma} \wh{F}(\gamma) \rho(\gamma)^* = s_{[\psi,\psi]} F = F \, ,
$$
which can be extended to the whole $\hh([\psi,\psi])$ by density.
For a relationship between the Zak transform and the global isometry $U_\Psi$ of Proposition \ref{prop:periodization} in the setting of LCA groups, see \cite[Prop. 6.7]{BHP15b}.
\end{rem}

\subsection{A two-pronged comb in $\ell_2(\Gamma)$}

In this subsection, we study  properties of a two-pronged comb $f$ of $\ell_2(\Gamma)$. We shall analyze when it generates the whole $\ell_2(\Gamma)$ and under which conditions the system $\{\lambda(\gamma)f \, : \, \gamma \in \Gamma\}$ 
has reproducing properties.  
 
To begin with, we recall that a two-pronged comb $f\in \ell_2(\Gamma)$ is a sequence of the form $f = a \delta_{\gamma_1} + b \delta_{\gamma_2} $  for $\gamma_1, \gamma_2 \in \Gamma$, with $\gamma_1 \neq \gamma_2$, and $a, b \in \C\setminus\{0\}$.  We denote by $V(f)$ the left-invariant space generated by $f$,  that is
$$
V(f) = \ol{\vsp\{\lambda(\gamma)f\}_{\gamma \in \Gamma}}^{\ell_2(\Gamma)} .
$$

The following lemma states conditions for  a two-pronged comb to generate $\ell_2(\Gamma)$.
\begin{lem}\label{lem:twoprongedcomb}
Let $\gamma_1, \gamma_2 \in \Gamma$, with $\gamma_1 \neq \gamma_2$,  $a, b \in \C\setminus\{0\}$ and $f = a \delta_{\gamma_1} + b \delta_{\gamma_2} $. Let $h = \gamma_1^{-1} \gamma_2$ and $e\in \Gamma$ the identity.
\begin{itemize}
\item[i)] If there is no $n \in \N$ such that $h^n = \id$, then $V(f) = \ell_2(\Gamma)$.
\item[ii)] If there exists $n \in \N$ such that $h^n = \id$, and $a \neq \pm b$, then $V(f) = \ell_2(\Gamma)$.
\end{itemize}
\end{lem}
\begin{proof}
Since the closed subspace $V(f)$ is left-invariant, by Theorem \ref{theo:Fourierinvariance} we have that $\F_\Gamma V(f) = \Proj_{V(f)} L^2(\vn(\Gamma))$. In particular, $\F_\Gamma f = \Proj_{V(f)} \F_\Gamma f$. Thus, the condition $V(f) = \ell_2(\Gamma)$ holds whenever $\Ker(\F_\Gamma f)^* = \{0\}$. Indeed. If $\Ker(\F_\Gamma f)^* = \{0\}$, we will have that 
$\Ker(\Proj_{V(f)}) = \{0\}$ because $(\F_\Gamma f)^* = (\F_\Gamma f)^*\Proj_{V(f)}$. Thus, $\Proj_{V(f)} = \Id_{\ell_2(\Gamma)}$ and therefore $V(f) = \ell_2(\Gamma)$. 

For computing  $\Ker(\F_\Gamma f)^*$, note that, since $(\F_\Gamma f)^* = \ol{a} \rho(\gamma_1) + \ol{b}\rho(\gamma_2)$, any $g \in \Ker(\F_\Gamma f)^*$ must satisfy
\begin{equation}\label{eq:twoprongedcombkernel}
(\F_\Gamma f)^* g(\gamma) = \ol{a} g(\gamma\gamma_1) + \ol{b}g(\gamma\gamma_2) = 0 \quad \forall \ \gamma \in \Gamma. 
\end{equation}

Now, let $g\in  \Ker(\F_\Gamma f)^*$ and suppose that   $g \neq 0$.  Choose $\gamma_0 \in \Gamma$ such that $g(\gamma_0) \neq 0$ and,  for $n \in \Z$, let $\gamma = \gamma_0 h^{n-1} \gamma_1^{-1}$. Then,  by (\ref{eq:twoprongedcombkernel}) we have that $$
0 = \ol{a} g(\gamma_0 h^{n-1}) + \ol{b}g(\gamma_0 h^n) 
$$ which is equivalent to $ g(\gamma_0 h^n) = -\frac{\ol{a}}{\ol{b}} g(\gamma_0 h^{n-1})$.
Thus,
\begin{equation}\label{eq:twoprongedcombtrick}
g(\gamma_0 h^n) = (-1)^n\left(\frac{\ol{a}}{\ol{b}}\right)^n g(\gamma_0).
\end{equation}

In case $i)$, all elements $\gamma_0 h^n$ are different, so using \eqref{eq:twoprongedcombtrick} we have 
$$
\|g\|^2_{\ell_2(\Gamma)} \geq \sum_{n \in \Z} |g(\gamma_0 h^n)|^2 = |g(\gamma_0)|^2 \sum_{n \in \Z}  \left|\frac{a}{b}\right|^{2n} = + \infty
$$
for any $a, b \in \C\setminus\{0\}$. Since $g \in \ell_2(\Gamma)$, this is a contradiction, thus $g = 0$, and $\Ker(\F_\Gamma f)^* = \{0\}$.

In case $ii)$, if $n \in \N$ is such that $h^n = \id$, then from (\ref{eq:twoprongedcombtrick}) we get
$$
g(\gamma_0) = g(\gamma_0 h^n) = (-1)^n\left(\frac{\ol{a}}{\ol{b}}\right)^n g(\gamma_0).
$$
Since $g(\gamma_0)\neq 0$, we then have that $(-1)^n\left(\frac{\ol{a}}{\ol{b}}\right)^n = 1$ and this is true only when $n$ is odd and $a = -b$ or when $n$ is even and $a = b$. As a consequence, if $a \neq \pm b$, we deduce that  $\Ker(\F_\Gamma f)^* = \{0\}$
\end{proof}

\begin{rem}
The condition $a \neq \pm b$ cannot be removed from item $ii)$ in Lemma \ref{lem:twoprongedcomb}. To see this, consider $\Gamma=\Z_2$. If $a\in\C\setminus\{0\}$ and $f=a(\delta_0+\delta_1)$ then, 
$V(f)=\textrm{span}\{\delta_0+\delta_1\}$ which is not $\ell_2(\Z_2)$. If $f=a(\delta_0-\delta_1)$ then, 
$V(f)=\textrm{span}\{\delta_0-\delta_1\}$ which is not $\ell_2(\Z_2)$.

\end{rem}
We now want to study the reproducing properties of $\{\lambda(\gamma)f \, : \, \gamma \in \Gamma\}$, with $f = a \delta_{\gamma_1} + b \delta_{\gamma_2} $ a two-pronged comb. In order to do so, we need to study the bracket map $[f,f]$ which reads, using (\ref{eq:bracketleft}),
\begin{align}\label{eq:twoprogedcombbracket}
[f,f] & = |\F_\Gamma f|^2 = |a\rho(\gamma_1)^* + b \rho(\gamma_2)^*|^2 = (\ol{a}\rho(\gamma_1) + \ol{b} \rho(\gamma_2))(a\rho(\gamma_1)^* + b \rho(\gamma_2)^*)\nonumber\\
& = (|a|^2 + |b|^2) \Id_{\ell_2(\Gamma)} + \ol{a}b \rho(\gamma_1 \gamma_2^{-1}) + \ol{b}a \rho(\gamma_1 \gamma_2^{-1})^* .
\end{align}

\begin{prop}
Let $f = a \delta_{\gamma_1} + b \delta_{\gamma_2} \in \ell_2(\Gamma)$ be a two-pronged comb, with $\gamma_1 \neq \gamma_2 \in \Gamma$ and $a,b \in \C\setminus\{0\}$. If $|a| \neq |b|$, the collection $\{\lambda(\gamma)f \, : \, \gamma \in \Gamma\}$ is a Riesz basis for $\ell_2(\Gamma)$.
\end{prop}

\begin{proof}
Observe first that, for all $\gamma \in \Gamma, a,b, \in \C$, both the operators
$$
Z^-(\gamma) = 2 |ab| \Id_{\ell_2(\Gamma)} - \ol{a}b \rho(\gamma) - \ol{b}a \rho(\gamma)^* \ , \quad Z^+(\gamma) = 2 |ab| \Id_{\ell_2(\Gamma)} + \ol{a}b \rho(\gamma) + \ol{b}a \rho(\gamma)^* 
$$
are positive. Indeed, $Z^-(\gamma) = X^*X$ with $X = \sqrt{|ab|} \Id_{\ell_2(\Gamma)} - \frac{a \ol{b}}{\sqrt{|ab|}} \rho(\gamma)^*$, while $Z^+(\gamma) = Y^*Y$ with $Y = \sqrt{|ab|} \Id_{\ell_2(\Gamma)} + \frac{a \ol{b}}{\sqrt{|ab|}} \rho(\gamma)^*$. Thus we can write
$$
[f,f] - Z^+(\gamma_1 \gamma_2^{-1}) \leq [f,f] 
\leq [f,f] + Z^-(\gamma_1 \gamma_2^{-1})
$$
which reads, by (\ref{eq:twoprogedcombbracket})
$$
(|a| - |b|)^2 \Id_{\ell_2(\Gamma)} \leq [f,f] \leq (|a| + |b|)^2 \Id_{\ell_2(\Gamma)} .
$$
By \cite[ii), Theorem A]{BHP15a}, when $|a| \neq |b|$, we then have that $\{\lambda(\gamma)f \, : \, \gamma \in \Gamma\}$ is a Riesz basis of $V(f)$, and by Lemma \ref{lem:twoprongedcomb} we have that $V(f) = \ell_2(\Gamma)$.
\end{proof}

\subsection{Dihedral action on $L^2(\R^2)$}

The smallest nonabelian group is $\Gamma = \SY_3$, the dihedral group of order 6, the symmetry group of an equilateral triangle. It is a group with 6 elements and 2 generators, which can be presented by
$$
\SY_3 = \langle a, b \, | \, a^3 = \id, b^2 = \id, ba = a^2b\rangle .
$$
We can write $\SY_3$ as a set in terms of the two generators $a$ and $b$ by
\begin{equation}\label{eq:S3order}
\SY_3 = \{\id,a,a^2,b,ab,a^2b\}.
\end{equation}
Following this order, the adjoint right regular representation is then given by
$$
\rho(a)^* = \left(
\begin{array}{cccccc}
0 & 0 & 1 & 0 & 0 & 0\\
1 & 0 & 0 & 0 & 0 & 0\\
0 & 1 & 0 & 0 & 0 & 0\\
0 & 0 & 0 & 0 & 1 & 0\\
0 & 0 & 0 & 0 & 0 & 1\\
0 & 0 & 0 & 1 & 0 & 0\\
\end{array}
\right) \, , \
\rho(b)^* = \left(
\begin{array}{cccccc}
0 & 0 & 0 & 1 & 0 & 0\\
0 & 0 & 0 & 0 & 1 & 0\\
0 & 0 & 0 & 0 & 0 & 1\\
1 & 0 & 0 & 0 & 0 & 0\\
0 & 1 & 0 & 0 & 0 & 0\\
0 & 0 & 1 & 0 & 0 & 0\\
\end{array}
\right)
$$
and their compositions.

Let $R_a \displaystyle = \left(\begin{array}{rr} -\frac12 & -\frac{\sqrt{3}}{2} \vspace{2pt}\\ \frac{\sqrt{3}}{2} & -\frac12 \end{array}\right)$ be the 120 degrees rotation on the plane, let $R_b \displaystyle = \left( \begin{array}{rr} 1 & 0 \\ 0 & -1\end{array}\right)$ be the reflection over the $x$ axis and, for $\gamma \in \SY_3$, let us denote by $R_\gamma$ the matrix obtained by the corresponding composition of these two matrices, e.g. $R_{ab} = R_aR_b$. Then we can define a representation $\pi : \SY_3 \to \mathcal{U}(L^2(\R^2))$ by $\pi(\gamma)f(x) = f(R_\gamma^{-1}x)$ for $f \in L^2(\R^2)$ and $\gamma \in \SY_3$.

\begin{figure}[h!]
\centering
\includegraphics[width=.52\textwidth]{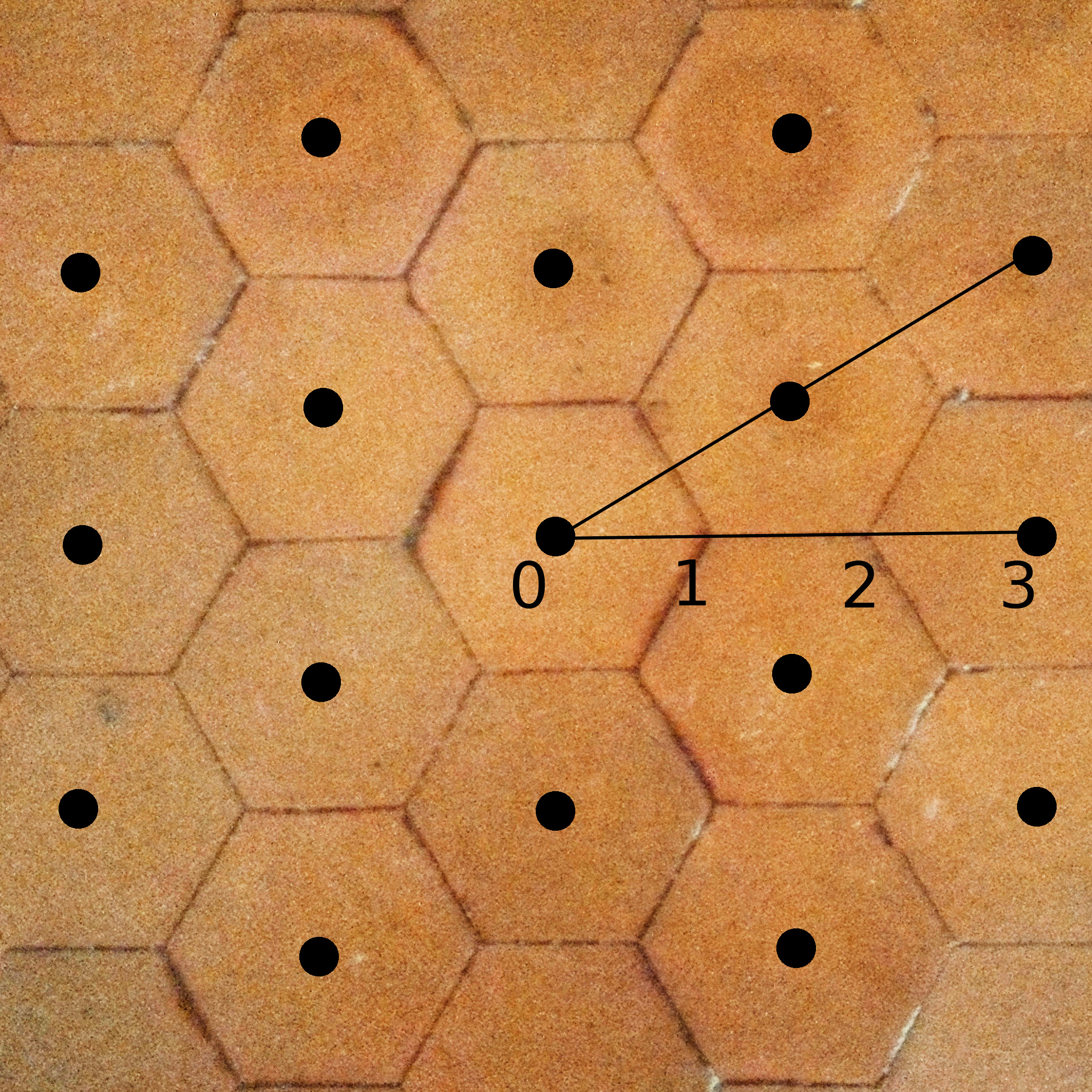}
\caption{Hexagonal lattice $\lat$ on the floor of the Maths department at the University of Buenos Aires.}\label{fig:hexagon}
\end{figure}

We want to provide a Helson map for this representation based on the construction given in Proposition \ref{prop:periodization}.
In order to do so, we start by choosing an orthonormal basis for $L^2(\R^2)$. Let $H \subset \R^2$ be the hexagonal domain with vertices
$$
(1,0)\, , \ (\frac12,\frac{\sqrt{3}}{2})\, , \ (-\frac12,\frac{\sqrt{3}}{2}) \, , \ (-1,0)\, , \ (-\frac12,-\frac{\sqrt{3}}{2})\, , \ (\frac12,-\frac{\sqrt{3}}{2})
$$
(see Figure \ref{fig:hexagon}), and let $\LL \displaystyle = \left( \begin{array}{rr} 3 & \frac32 \vspace{2pt}\\ 0 & \frac{\sqrt{3}}{2} \end{array}\right)$. Then $H$ tiles $\R^2$ by translations with the lattice $\displaystyle\lat = \LL \Z^2 = \bigg\{(3m + \frac32 n, \frac{\sqrt{3}}{2}n) \, : \, (m,n) \in \Z^2\bigg\}$. Let us denote by $\wl = (\LL^t)^{-1} \displaystyle = \left( \begin{array}{rr} \frac13 & 0 \vspace{2pt}\\ -\frac{1}{\sqrt{3}} & \frac{2}{\sqrt{3}} \end{array}\right)$, and by $\lat^\bot = \{k \in \R^2 \, : \, k \cdot l \in \Z \ \forall \, l \in \lat\} = \wl \Z^2$ the annihilator lattice of $\lat$. Then it is well known \cite{Fug74} that $\{\frac{1}{\sqrt{|H|}}e^{2\pi i k \cdot}\}_{k \in \lat^\bot}$ is an orthonormal basis of $L^2(H)$, where $|H| = \frac{3\sqrt{3}}{2}$. Thus, the system $\Psi = \{\psi_{l,k} \, : \, (l,k) \in \lat \times \lat^\bot\} \subset L^2(\R^2)$ given by
$$
\psi_{l,k}(x) = \frac{1}{\sqrt{|H|}} T_l e^{2\pi i k \cdot x} \one{H}(x) = \frac{1}{\sqrt{|H|}} e^{2\pi i k \cdot x} \one{H + l}(x)
$$
defines an orthonormal basis of $L^2(\R^2)$, and we will use it to define the family of Lemma \ref{lem:generators}.

Since $H$ is invariant under rotations of 120 degrees and reflections over the $x$ axis, and since each $R_\gamma$ is an orthogonal matrix,
\begin{align*}
\pi(\gamma)\psi_{l,k}(x) & = \frac{1}{\sqrt{|H|}} e^{2\pi i k \cdot R_\gamma^{-1} x} \one{H + l}(R_\gamma^{-1} x) = \frac{1}{\sqrt{|H|}} e^{2\pi i (R_\gamma^{-1})^t k \cdot x} \one{R_\gamma(H + l)}(x)\\
& = \frac{1}{\sqrt{|H|}} e^{2\pi i R_\gamma k \cdot x} \one{H + R_\gamma l}(x) = \psi_{R_\gamma l, R_\gamma k}(x).
\end{align*}
Notice that $(R_\gamma l, R_\gamma k) \in \lat \times \lat^\bot$ for all $(l,k) \in (\lat \times \lat^\bot)$, because $R_\gamma l = L (L^{-1}R_\gamma l)$ and $L^{-1}R_\gamma l \in \Z^2$ for all $l \in \lat$, and the same holds for $\lat^\bot$. Thus
$$
\pi(\gamma)\psi_{l,k} \in \Psi \quad \forall \ \gamma \in \gamma \, , \ \forall \ (l,k) \in \lat \times \lat^\bot .
$$
Let us call $r$ the representation of $\SY_3$ in $\lat \times \lat^\bot$ given by $r_\gamma(l,k) = (R_\gamma l, R_\gamma k)$.
Then the set
$$
\ind = \Big(\lat \cap \{(x,y) \in \R^2 \, : \, 0 \leq y < \sqrt{3} x\} \Big) \times \Big(\lat^\bot \cap \{(x,y) \in \R^2 \, : \, 0 \leq y < \sqrt{3} x\} \Big)
$$
is a section of $(\lat \times \lat^\bot)/r$, i.e.
$
\displaystyle \lat \times \lat^\bot = \bigcup_{(l,k) \in \ind} \{r_\gamma (l,k) \, : \, \gamma \in \SY_3\} 
$
as a disjoint union, so that
$$
L^2(\R^2) = \bigoplus_{(l,k) \in \ind} \langle\psi_{(l,k)}\rangle_{\SY_3}
$$
where $\langle\psi_{l,k}\rangle_{\SY_3}$ is actually the finite span of the orbit $\{\pi(\gamma)\psi_{l,k}\}_{\gamma \in \SY_3}$.
Let us write $\ind$ as the disjoint union
$$
\ind = \{(0,0)\} \cup \partial \ind \cup \mathring\ind
$$
where
$$
\partial \ind = \Big\{(3m,0) \, : \, m = 1, 2, \dots \Big\} \times \Big\{\Big(\frac23 m, 0\Big) \, : \, m = 1, 2, \dots\Big\}
$$
and
$$
\mathring\ind = \Big(\lat \cap \{(x,y) \in \R^2 \, : \, 0 < y < \sqrt{3} x\} \Big) \times \Big(\lat^\bot \cap \{(x,y) \in \R^2 \, : \, 0 < y < \sqrt{3} x\} \Big).
$$
Notice that $r_\gamma(0,0) = (0,0)$ for all $\gamma \in \SY_3$, $r_b(l,k) = (l,k)$ for all $(l,k) \in \partial \ind$, and $r_\gamma(l,k) \neq r_{\gamma'}(l,k)$ for all $\gamma, \gamma' \in \SY_3, \gamma \neq \gamma'$ and all $(l,k) \in \mathring \ind$.

Since $\SY_3$ is finite, for all $p \geq 1$ we have $L^p(\vn(\SY_3)) = \vn(\SY_3) \approx M_{6 \times 6}(\C)$, so the bracket map writes as the finite sum
$$
[\varphi,\psi] = \sum_{\gamma \in \SY_3} \langle \varphi, \pi(\gamma)\psi \rangle_{L^2(\R^2)} \,\rho(\gamma)^* .
$$
Using that $\pi(\gamma)\psi_{l,k} = \psi_{r_\gamma(l,k)}$, and by the orthonormality of $\Psi$, we get
\begin{align*}
[\psi_{0,0},\psi_{0,0}] & = \sum_{\gamma \in \SY_3} \rho(\gamma)^*\\
[\psi_{l,k},\psi_{l,k}] & = \sum_{\gamma \in \SY_3} \langle \psi_{l,k}, \psi_{r_\gamma (l,k)} \rangle_{L^2(\R^2)} \,\rho(\gamma)^* = \Id_{\C^6} + \rho(b)^* \quad \forall \ (l,k) \in \partial \ind\\
[\psi_{l,k},\psi_{l,k}] & = \Id_{\C^6} \quad \forall \ (l,k) \in \mathring\ind .
\end{align*}
Note that $\sum_{\gamma \in \SY_3} \rho(\gamma)^*$ is the $6 \times 6$ matrix with $1$ in all entries, that is 6 times a projection of rank 1 in $\C^6$, while $\Id_{\C^6} + \rho(b)^* = \Id_{\C^6} + \rho(b) = \frac12(\Id_{\C^6} + \rho(b))^2$ is twice a projection of rank $3$ in $\C^6$.
Then, we have that
\begin{itemize}
\item $\{\pi(\gamma)\psi_0\}_{\gamma \in \SY_3}$ is a tight frame with constant 6;
\item $\{\pi(\gamma)\psi_j\}_{\gamma \in \SY_3}$, for $j \in \partial \ind$, is a tight frame with constant 2;
\item $\{\pi(\gamma)\psi_j\}_{\gamma \in \SY_3}$, for $j \in \mathring\ind$, is an orthonormal system.
\end{itemize}

We can then compute the Helson map $U_\Psi$ of Proposition \ref{prop:periodization} as follows:
\begin{align*}
U_\Psi[\varphi]_{0,0} & = \frac{1}{\sqrt{6}}[\psi_0,\psi_0] \frac16 [\varphi,\psi_{0,0}] = \frac{1}{6\sqrt{6}}\sum_{\gamma \in \SY_3} \rho(\gamma)^*[\varphi,\psi_{0,0}]\\
& = \frac{1}{6\sqrt{6}}\sum_{\gamma \in \SY_3} [\varphi,\pi(\gamma)\psi_{0,0}] = \frac{1}{\sqrt{6}}[\varphi,\psi_{0,0}]\\
U_\Psi[\varphi]_{l,k} & = \frac{1}{\sqrt{2}}[\psi_{l,k},\psi_{l,k}] \frac16 [\varphi,\psi_{l,k}] = \frac{1}{\sqrt{2}} (\Id_{\C^6} + \rho(b)^*) \frac12 [\varphi,\psi_{l,k}]\\
& = \frac{1}{2\sqrt{2}} \left([\varphi,\psi_{l,k}] + \rho(b)^* [\varphi,\psi_{l,k}]\right) = \frac{1}{2\sqrt{2}} \left([\varphi,\psi_{l,k}] + [\varphi,\pi(b)\psi_{l,k}]\right)\\
& = \frac{1}{\sqrt{2}} [\varphi,\psi_{l,k}]\ , \quad (l,k) \in \partial \ind\\
U_\Psi[\varphi]_{l,k} & = [\varphi,\psi_{l,k}] \ , \quad (l,k) \in \mathring \ind .
\end{align*}

\subsection{Translates for number-theoretic groups}

It is well known the there are LCA groups having no discrete subgroups and therefore, they do not fit in the setting of Section \ref{sec:translates} for analyzing spaces invariant under translations neither reproducing properties. In order to overcome this obstacle J. Benedetto and R. Benedetto proposed the following setting where a new kind of translation operators are defined (see \cite{BB04, BB18}). 

Let $G$ be a number-theoretic group, that is an LCA group with a compact and open subgroup $H$. Assume that $G$ is second countable and fix $\mathcal{C}\subset \wh G$ a section for the quotient $\wh G/H^\perp$, which turns out to be discrete and countable. We denote by $\wh f(\gamma) = \int_G f(x) \ol{\gamma(x)} dx$  the Fourier transform in the LCA group $G$.
The {\it translation operator by an element $[x]\in G/H$} of a function $f\in L^2(G)$ is noted by $T_{[x]}$
and defined through its Fourier Transform as 
$$\wh{T_{[x]}f} = \wh f \omega_{[x]},$$
where $\omega_{[x]}: \wh G\to \C$ is given by $\omega_{[x]}(\gamma):= \overline{\eta_\gamma(x)}$ for $\gamma = \eta_\gamma + \sigma_\gamma$ with $\eta_\gamma\in H^\perp$ and $\sigma_\gamma\in \mathcal{C}$.
These translation operators give rise to a unitary representation 
of the discrete group $G/H$ on $ L^2(G)$, namely $T:G/H\to \mathcal{U}( L^2(G))$, $[x]\mapsto T_{[x]}$. Indeed. By \cite[Rem. 2.3]{BB04} it holds that $T_{[x]}T_{[y]} = T_{[x+y]}$ for all $[x], [y]\in G/H$ and that  $T_{[e]} = \Id_{L^2(G)}$. Moreover, since $|\omega_{[x]}| = 1$ we have that 
$$\|T_{[x]}f\|_{L^2(G)} = \|\wh f \omega_{[x]}\|_{L^2(\wh G)} = \|\wh f \|_{L^2(\wh G)} = \|f \|_{L^2(G)}.$$ 

Let us see that $(G/H, T, L^2(G))$ is a dual integrable triple. For this, let $f,g\in L^2(G)$ and $[x]\in G/H$. Then, 
\begin{align*} 
\langle f , T_{[x]} g\rangle_{L^2(G)} &= \int_{\wh G} \wh f (\gamma)\overline{ \wh g(\gamma) \omega_{[x]}(\gamma)}\,d\gamma = \sum_{\sigma\in\mathcal{C}}\int_{H^\perp +\sigma} \wh f (\gamma)\overline{ \wh g(\gamma) \omega_{[x]}(\gamma)}\,d\gamma \\
& = \sum_{\sigma\in\mathcal{C}}\int_{H^\perp } \wh f (\eta+\sigma)\overline{ \wh g(\eta+\sigma) \omega_{[x]}(\eta+\sigma)}\,d\eta \\
&= \int_{H^\perp }  \sum_{\sigma\in\mathcal{C}}\wh f (\eta+\sigma)\overline{ \wh g(\eta+\sigma) }\eta(x)\,d\eta
\end{align*}
where we have used Plancherel Theorem, that $\wh G$ can be partitioned  by $\{H^\perp+\sigma\}_{\sigma\in\mathcal{C}}$ and the definition of $\omega_{[x]}$.
Since clearly $\sum_{\sigma\in\mathcal{C}}\wh f (\cdot+\sigma)\overline{ \wh g(\cdot+\sigma) } \in L^1(H^\perp)$, and $H^\perp\approx \wh {G/H}$, we conclude that the bracket map is given by 
\begin{equation}\label{eq:bracket-numer-theoretic-groups}
[f,g](\eta) = \sum_{\sigma\in\mathcal{C}}\wh f (\eta+\sigma)\overline{ \wh g(\eta+\sigma) }\quad\textrm{ for a.e. }\eta\in H^\perp.
\end{equation}
In this context, it can be proven that  the mapping given by 
$$\iso:L^2(G)\to L^2(H^\perp, \ell^2(\mathcal{C})), \quad\iso[f](\eta):=\{\wh f(\eta+\sigma) \}_{\sigma\in\mathcal{C}}
$$ for a.e. $\eta \in H^\perp$ is an isometric isomorphism that satisfies $\iso[T_{[x]} f] (\eta) = \eta(x) \iso[ f] (\eta) $ for a.e. $\eta \in H^\perp$. Thus, it is a Helson map for $(G/H, T, L^2(G))$. 

Recently, in \cite[Th. 4.5]{BB18}, it was proven that for $f\in L^2(G)$, the family $\{T_{[x]} f: [x]\in G/H\}$ is a frame sequence with constants $0<A\leq B<\infty$ if and only if 
$$ A\leq \sum_{\sigma\in\mathcal{C}}|\wh f (\eta+\sigma)|^2\leq B,$$
for a.e. $\eta \in \{\eta\in H^\perp: \sum_{\sigma\in\mathcal{C}}|\wh f (\eta+\sigma)|^2\neq0\}$. 
Once we have proven that $(G/H, T, L^2(G))$ is a dual integrable triple, one sees that \cite[Th. 4.5]{BB18} is the version of \cite[Th. A]{BHP15a} applied to this context (see also Corollary \ref{cor:principal} and \cite[Sec. 5]{BHP17}). Moreover, our Theorem \ref{theo:frames}
 generalizes \cite[Th. 4.5]{BB18} for families of the form $\{T_{[x]}\phi_i: [x]\in G/H, \,i\in \ind\}$ where $\ind $ is an at most countable index set. 

\

\noindent
{\bf Acknowledgements}: The authors would like to thank the anonymous referee for his/her comments which helped us to improve the presentation of the paper.

\bibliographystyle{abbrv}
\bibliography{BIBLIO.bib}

\noindent
\textbf{Davide Barbieri}\\
Universidad Aut\'onoma de Madrid, 28049 Madrid, Spain\\
\href{mailto:davide.barbieri@uam.es}{\tt davide.barbieri@uam.es}

\

\noindent
\textbf{Eugenio Hern\'andez}\\
Universidad Aut\'onoma de Madrid, 28049 Madrid, Spain\\
\href{mailto:eugenio.hernandez@uam.es}{\tt eugenio.hernandez@uam.es}

\

\noindent
\textbf{Victoria Paternostro}\\
Universidad de Buenos Aires and 
IMAS-CONICET, Consejo  Nacional de Investigaciones Cient\'ificas y T\'ecnicas, 1428 Buenos Aires,  Argentina\\
\href{mailto:vpater@dm.uba.ar}{\tt vpater@dm.uba.ar}

\end{document}